\documentclass[12pt]{amsart}
\advance\hoffset by -0.5 truecm
\usepackage{amsmath,amscd}
\usepackage{amssymb}
\usepackage{amsthm}
\usepackage{color}
\usepackage[normalem]{ulem}
\usepackage{cancel}

\newtheorem{Theorem}{Theorem}[section]

\newtheorem{Lemma}[Theorem]{Lemma}

\newtheorem{Proposition}[Theorem]{Proposition}

\newtheorem{Remark}[Theorem]{Remark}

\def\Z{{\mathbb Z}}

\def \R{{\mathbb R}}

\def \T{{\mathbb T}}

\def \S{{\mathbb S}}

\newcommand{\ep}{\varepsilon}
\newcommand{\N}{\mathbb{N}}

\newcommand{\set}[2]{\left\{{#1}\mid{#2}\right\}}

\title[On the existence of three closed magnetic geodesics]{On the existence of three closed magnetic geodesics for subcritical energies}

\author[A. Abbondandolo]{Alberto Abbondandolo}
\address{Ruhr Universit\"at Bochum,
Fakult\"at f\"ur Mathematik,
Geb\"aude NA 4/33,
D-44801 Bochum,
Germany}
\email{alberto.abbondandolo@rub.de}

\author[L. Macarini]{Leonardo Macarini}
\address{Universidade Federal do Rio de Janeiro, Instituto de Matem\'atica,
Cidade Universit\'aria, CEP 21941-909 - Rio de Janeiro - Brazil}
\email {leonardo@impa.br}

\author[G.P. Paternain]{Gabriel P. Paternain}
\address{ Department of Pure Mathematics and Mathematical Statistics,
University of Cambridge,
Cambridge CB3 0WB, UK}
\email {g.p.paternain@dpmms.cam.ac.uk}

\begin{document}

\begin{abstract} We consider exact magnetic flows on closed orientable surfaces. We show that for almost every energy $\kappa$ below Ma\~n\'e's critical value of the universal covering there are always at least three distinct closed magnetic geodesics with energy $\kappa$. If in addition the energy level is assumed to be non-degenerate we prove existence of infinitely many closed magnetic geodesics.

\end{abstract}

\maketitle


\renewcommand{\theenumi}{\roman{enumi}}
\renewcommand{\labelenumi}{(\theenumi)}

\section*{Introduction} In this paper we study the problem of existence of closed orbits with prescribed energy for exact magnetic flows on closed orientable surfaces.

Let $(M,g)$ be a closed oriented Riemannian surface and let $\theta$ be a smooth $1$-form on $M$. Set $f:=\star d\theta$, where $\star$ is the Hodge star operator of $g$. The {\it magnetic geodesics} $\gamma:\R\to M$ are the solutions to
the ordinary differential equation
\[\frac{D\dot{\gamma}}{dt}=f(\gamma)\,i\dot{\gamma},\]
where $D/dt$ is the covariant derivative and $i$ is the almost complex structure
naturally associated to the oriented Riemannian surface. Magnetic geodesics were introduced by V.I.~Arnold in \cite{arn61b}. Certainly, the solutions to this ODE preserve the kinetic energy $E(x,v):=|v|_{x}^2/2$.

The above ODE is the Euler-Lagrange equation of the Lagrangian $L:TM\to \R$ given by
\[L(x,v)=\frac{1}{2}|v|^2_{x}+\theta_{x}(v),
\]
which is fiberwise strictly convex and superlinear, and hence satisfies the standard hypotheses in Aubry-Mather theory. The corresponding Hamiltonian on $T^*M$ is given by 
\[
H(x,p)=\frac{1}{2}|p-\theta_{x}|_{x}^2.
\] 

We are interested in closed magnetic geodesics with prescribed energy $\kappa$, and hence it is useful to have a variational principle that picks up these orbits. This is easily achieved by considering the {\it free-period action functional}
\[
\S_{\kappa} (\gamma) := \int_0^T \bigl( L(\gamma(t),\dot{\gamma}(t)) + \kappa \bigr)\, dt = T \int_{\T}  \bigl( L(x(s),\dot{x}(s)/T) + \kappa \bigr)\, ds =: \S_{\kappa} (x,T),
\]
where $\gamma:[0,T]\rightarrow M$ is an absolutely continuous closed curve and $x(s):= \gamma(Ts)$ is its reparametrization on $\T:= \R/\Z$. This functional is smooth on the Hilbert manifold $\mathcal{M}:= H^1(\T,M)\times (0,+\infty)$, and its critical points correspond to periodic magnetic geodesics of energy $\kappa$ (as usual, $H^1(\T,M)$ denotes the Hilbert manifold of absolutely continuous closed loops $x:\T \rightarrow M$ whose derivative is square integrable). Elements of $\mathcal{M}$ are denoted indifferently as closed curves $\gamma:[0,T]\rightarrow M$ or as pairs $(x,T)$, as above.

In \cite{cmp04} the authors proved that every energy level $E^{-1}(\kappa)$ ($\kappa\geq 0$) contains a closed magnetic geodesic. The proof hinges on the consideration of certain critical energies known as {\it Ma\~n\'e's critical values}.
For our purposes there are two of these critical energies which are significant, and they are defined as follows:
\[
\begin{split}
c_u := \inf \{ \kappa\in \R \, | \, & \S_{\kappa}(\gamma)\geq 0 \mbox{ for every absolutely continuous} \\ & \mbox{contractible closed curve } \gamma\},
\end{split}
\]
\[
\begin{split}
c_0 := \inf \{ \kappa\in \R \,  | \, &\S_{\kappa}(\gamma)\geq 0 \mbox{ for every absolutely continuous} \\ & \mbox{null-homologous closed curve } \gamma\}.
\end{split}
\]
Clearly $c_u\leq c_0$ and equality holds for $M=S^2$ and $M=\T^2$. However for surfaces of higher genus the inequality could be strict \cite{pp97}. If $\kappa>c_u$ it is possible to show that ${\S_{\kappa}}$ satisfies the Palais-Smale condition and it is bounded from below on each connected component of the free loop space \cite{cipp00}. Thus one is able to derive the same existence results as in the case of the geodesic flow of a Finsler metric. In particular every non-trivial free homotopy class contains a closed magnetic geodesic with energy $\kappa>c_u$.  Actually, for $\kappa>c_0$ the flow on $E^{-1}(\kappa)$ can be reparametrized to a Finsler geodesic flow. In particular, when $M=S^2$ the flow for energies above $c_u=c_0$ has always at least two closed magnetic geodesics by \cite{bl10}. Katok's example from \cite{kat73} shows that in this case there may be only two closed magnetic geodesics.

If $\kappa$ belongs to the interval $(c_u,c_0]$, then there are infinitely many closed magnetic geodesics of energy $\kappa$. Indeed, this interval is non-empty only if $M$ has genus greater than 1, and in this case $\mathcal{M}$ has infinitely many connected components with the property that no element of one of these is the iteration of an element of another one. In this case, we can find infinitely many closed magnetic geodesics with energy $\kappa>c_u$ by minimizing $\S_{\kappa}$ on each of these connected components.

In this paper we are concerned with the more difficult range of subcritical energies in the interval $(0,c_u)$. What makes this range of energies harder is that there are basically no tools to tackle it: the free-period action functional is no longer bounded from below and it may not satisfy the Palais-Smale condition. One is tempted then to deploy the Symplectic Topology arsenal, but unfortunately, as it was shown in \cite{cmp04}, the energy levels below $c_0$ are never of contact type at least for $M\neq \T^2$.

The idea used in \cite{cmp04} to produce a closed magnetic geodesic goes back to I. Taimanov \cite{tai92b,tai92c,tai92} (who proved similar results but with different methods) and consists in considering just simple closed curves to make the action functional bounded from below. Technically the space of simple closed curves is not the best to work with but by considering
integral 2-currents with suitable multiplicity one can use the compactness and regularity results from Geometric Measure Theory. In the end one obtains for $\kappa<c_0$ a closed magnetic geodesic $\alpha_{\kappa}$ which has {\it negative} ${\S_{\kappa}}$-action. This orbit has also the property of being a local $C^1$-minimizer. To obtain this orbit one may need to pass to a finite cover
(see Section \ref{sec:prelim}) but this is not much of a problem.

The other result that is available in the subcritical range is due to G. Contreras \cite{con06}. He proved that for almost every $\kappa\in (0,c_u)$ there is a closed contractible orbit $\gamma_{\kappa}$ with energy $\kappa$ and {\it positive} ${\S_{\kappa}}$-action (this holds in any dimension and for any Tonelli Lagrangian).
The fact that $\gamma_{\kappa}$ and $\alpha_{\kappa}$ have actions with different signs implies they must be geometrically distinct (i.e. one is not an iterate of the other) and hence we deduce that for almost every $\kappa\in (0,c_u)$
there are at least {\it two} closed magnetic geodesics with energy $\kappa$.

The main purpose of the present paper is to upgrade this to almost everywhere existence of at least {\it three} closed magnetic geodesics. Moreover, assuming in addition that the energy level is non-degenerate we shall prove the existence of infinitely many closed orbits with energy $\kappa\in (0,c_u)$.

\medskip

\noindent{\bf Theorem.} {\it Let $(M,g)$ be a closed oriented Riemannian surface and let $\theta$ be a smooth 1-form.  Then there exists a full measure set $K\subset (0,c_u)$ such that for every $\kappa\in K$
 there are at least three closed magnetic geodesics with energy $\kappa$.   
 Moreover, if for $\kappa\in K$ the energy level $E^{-1}(\kappa)$ is non-degenerate, then there are infinitely many closed magnetic geodesics with energy $\kappa$.  }

\medskip

Let us clarify the meaning of ``non-degenerate energy level". A closed orbit  $\gamma=(x,T)$ of energy $\kappa$ is said to be transversally non-degenerate if the algebraic multiplicity of the eigenvalue 1 of $d\varphi_T(\gamma(0),\dot\gamma(0))$ is exactly two, where $\varphi_t:TM\to TM$ denotes the Lagrangian flow of $L$. This is equivalent to the fact that the linearized Poincar\'e map associated to a transverse section to the orbit $(\gamma,\dot\gamma)$ in $E^{-1}(\kappa)$ does not have the eigenvalue 1, and is also equivalent to the fact that
the second differential of $\S_{\kappa}$ restricts to a non-degenerate form on the normal bundle of 
\[
\T \cdot (x,T) := \set{(x(\sigma+\cdot),T)}{\sigma\in \T}.
\]
Transversal non-degeneracy implies that there is a smooth family $\{(x_{\kappa+\sigma},T(\kappa+\sigma))\}_{\sigma\in(-\varepsilon,\varepsilon)}$ of critical points of $\S_{\kappa+\sigma}$ (an orbit cylinder)  such that $(x_{\kappa},T(\kappa)) = (x,T)$ (see \cite[Theorem 8.2.2]{am78} or \cite[Proposition 4.2]{hz94}). 
The energy level $E^{-1}(\kappa)$ is said to be non-degenerate if every closed orbit on it is transversally non-degenerate.

Non-degeneracy is $C^r$-generic for every $1\leq r \leq \infty$: for every $\kappa>0$ there exists a subset of the space of exact 2-forms $d\theta$ which is residual in the $C^r$ topology for which $E^{-1}(\kappa)$ is non-degenerate (see \cite[Theorem 1.2]{mir06}). 

Thanks to results of J.A.G.~Miranda's, our main theorem implies that if $\kappa$ belongs to the full-measure set $K\subset (0,c_u)$ which is associated to the magnetic form $d\theta$, then there is an exact 2-form $d\tilde{\theta}$ with $\|d\tilde\theta - d\theta\|_{C^1}<\epsilon$ such that the exact magnetic flow associated to $d\tilde{\theta}$ on the energy level $\kappa$ has positive topological entropy: if $E^{-1}(\kappa)$ is non-degenerate, then this follows from the existence of infinitely many closed orbits, thanks to \cite[Theorem 1.2]{mir07}; if $E^{-1}(\kappa)$ is not non-degenerate, then it contains a non-hyperbolic closed orbit, and the existence of $d\tilde{\theta}$ follows from \cite[Theorem 1.1]{mir07} (when $E^{-1}(\kappa)$ has at least one non-hyperbolic closed orbit, $d\tilde{\theta}$ can be chosen to be close to $d\theta$ in the $C^{\infty}$ topology).

Let us briefly describe the idea of the proof of the Theorem. Since we are working on an orientable surface the iterates $\alpha^n_{\kappa}$ of $\alpha_{\kappa}$ are also local minimizers.  Using that  the action of $\alpha_{\kappa}$ is negative and ideas in \cite{ban80} we construct for all sufficiently large values of $n$ an appropriate negative minimax value and via Struwe's monotonocity argument \cite{str90} applied to the free-period action functional we prove the existence for a.e.\ $\kappa\in (0,c_u)$ of a closed orbit $\beta_{\kappa,n}$ with energy $\kappa$, negative ${\S_{\kappa}}$-action and which is not a strict local minimizer. These facts ensure that $\beta_{\kappa,n}$ is not an iterate of $\alpha_{\kappa}$ 
or $\gamma_{\kappa}$. Struwe's monotonicity argument is used to bypass the lack of the Palais-Smale condition as in \cite{con06}.
In the non-degenerate situation we show that the mean index of the minimax orbits $\beta_{\kappa,n}$ is actually positive which excludes the possibility that the curves $\beta_{\kappa,n}$ for $n\in \N$ are all iterates of finitely many simple periodic orbits.  To prove positivity of the mean index we use the following fact which is proved in Section \ref{Sec:meanindex} and has independent interest:  if a transversally non-degenerate closed orbit has $T'(\kappa)\geq 0$, then it must have positive mean index (or equivalently there are conjugate points along the orbit). The tools used in the proof of this fact allow us to prove that the mountain pass closed orbits are either non-hyperbolic or odd hyperbolic (see Proposition \ref{elliptic-odd_hyp} and Remark \ref{stability}). 

In \cite[Theorem 4]{ban80} Bangert shows that a Riemannian metric on $S^2$ which possesses a ``waist" must have infinitely many closed geodesics.  A waist is a closed geodesic which is a local minimum of
the energy functional. The situation that we have  for subcritical energies is similar: 
as we mentioned above $\alpha_{\kappa}$ is a local minimizer of the free period action functional so at this stage it seems reasonable to conjecture that for almost every $\kappa<c_u$ there are infinitely many closed magnetic geodesics with energy $\kappa$, regardless of any non-degeneracy assumption. 
We hope to address this question in a subsequent paper.

\medskip

\paragraph{\sc Acknowledgments.} 
We are grateful to M.~Mazzucchelli for communicating us the elegant proof of Lemma \ref{minimizers}.

The present work is part of the first author's activities within CAST, a Research Network Program of the European Science Foundation. The second author was partially supported by CNPq, Brazil.
This research has been supported partially by EU Marie-Curie IRSES Brazilian-European partnership in Dynamical Systems (FP7-PEOPLE-2012-IRSES 318999 BREUDS).

\section{Mean index of the free-period action functional}\label{Sec:meanindex}

The contents of this section hold for arbitrary Tonelli Lagrangians on configuration spaces of arbitrary dimension.

Let $M$ be a smooth orientable manifold and $H: T^*M \to \R$ a Hamiltonian dual to a Tonelli Lagrangian $L: TM \to \R$ (that is, $L$ is fiberwise superlinear and satisfies $d_{vv} L(q,v)>0$) via the Legendre transform $\mathfrak{L}:TM\rightarrow T^{*}M$. Consider the free period action functional
$$ \S_{\kappa}(x,T) = T\int_\T \bigl( L(x(s),\dot x(s)/T) + \kappa \bigr)\,ds, \qquad (x,T) \in C^{\infty}(\T,M) \times (0,+\infty). $$
If $(x,T)$ is a critical point of $\S_{\bar\kappa}$ and $\gamma(t)=x(t/T)$ is the corresponding $T$-periodic solution of the Euler-Lagrange equation associated to $L$, we define
\[
z(t) := \mathfrak{L}(\gamma(t),\dot{\gamma}(t)).
\]
Then $H(z(t)) \equiv \bar\kappa$, and $z$ is a $T$-periodic orbit of the Hamiltonian vector field $X_H$ on $T^*M$, which is defined by
\[
\imath_{X_H} \omega = - dH,
\]
where $\omega=dp\wedge dq$ is the standard symplectic form on $T^*M$. We denote by $\phi_t$ the flow of $X_H$ on $T^*M$.

The vector $X_H(z(0))$ is an eigenvector with eigenvalue 1 of the differential of the flow 
\[
d\phi_T(z(0)) : T_{z(0)} T^*M \rightarrow T_{z(0)} T^*M.
\]
Since the above linear map is symplectic, the eigenvalue 1 has algebraic multiplicity at least two. Let us assume that $(x,T)$ is {\em transversally non-degenerate}, meaning that the algebraic multiplicity of the eigenvalue 1 of $d\phi_T(z(0))$ is exactly two. In this case, the closed orbit $z$ admits a (unique) orbit cylinder: there exists a smooth map 
\[
(\kappa,t) \mapsto z_{\kappa}(t), \qquad (\kappa,t) \in (\bar\kappa-\epsilon,\bar\kappa+\epsilon) \times \R,
\]
which is unique up to time-shifts, such that each $z_{\kappa}$ is a closed orbit of $X_H$ with period $T(\kappa)$ on $H^{-1}(\kappa)$, and $z_{\bar\kappa}=z$, $T(\bar\kappa)=T$ (see e.g. \cite[Theorem 8.2.2]{am78} or \cite[Proposition 4.2]{hz94}).

The presence of this orbit cylinder allows to decompose $d\phi_T(z(0))$ in the following way (see also \cite[pages 104--105]{mp11}). Set
\[
\zeta := - \frac{\partial}{\partial \kappa} z_{\kappa}(0) \Big|_{\kappa=\bar\kappa}.
\]
By differentiating the identity $H(z_{\kappa}(0))=\kappa$ with respect to $\kappa$ at $\bar\kappa$, we obtain
\[
-dH(z(0))[\zeta]=1.
\]
Therefore
\[
\omega(X_H(z(0)),\zeta) = - dH(z(0))[\zeta] = 1,
\]
so $X_H(z(0))$ and $\zeta$ form a symplectic basis of the plane $V$ spanned by these two vectors. By differentiating the equation
\[
\phi_{T(\kappa)} (z_{\kappa}(0)) = z_{\kappa}(0)
\]
with respect to $\kappa$ at $\bar\kappa$, we obtain the identity
\[
T'(\bar\kappa) X_H(z(0)) - d \phi_{T}(z(0)) [\zeta] = - \zeta,
\]
which shows that the plane $V$ is $d \phi_{T}(z(0))$-invariant and that the restriction of $d \phi_{T}(z(0))$ to $V$ is represented by the matrix
\[
\left( \begin{array}{cc} 1 & T'(\bar\kappa) \\ 0 & 1 \end{array} \right)
\]
with respect to the symplectic basis $X_H(z(0)),\zeta$. Then the $\omega$-orthogonal complement $W$ of $V$ in $T_{z(0)} T^*M$ is also $d \phi_{T}(z(0))$-invariant. Being $\omega$-orthogonal to $X_H$, $W$ is contained in the kernel of $dH(z(0))$, so it is the tangent space of a Poincar\'e section, that is a hypersurface in $H^{-1}(\bar\kappa)$ which is transverse to the flow of $X_H$ at $z(0)$. If we denote by $P:W \rightarrow W$ the differential at $z(0)$ of the Poincar\'e return map associated to such a section, we see that $d \phi_{T}(z(0))$ has the form
\begin{equation}
\label{diffe}
d \phi_{T}(z(0)) = \left( \begin{array}{ccc} 1 & T'(\bar\kappa) & 0 \\ 0 & 1  & 0 \\ 0 & 0 & P \end{array} \right)
\end{equation}
with respect to the symplectic decomposition $T_{z(0)} T^*M = V \oplus W$. Therefore, the transversal non-degeneracy condition is equivalent to the fact that the linearized Poincar\'e map $P$ does not have the eigenvalue 1. From this, it can also be shown that $(x,T)$ is transversally non-degenerate if and only if the kernel of the second Gateaux differential of the free period action functional $\S_{\bar\kappa}$ at $(x,T)$ is one-dimensional (hence spanned by the vector $(\dot{x},0)$).

Denote by $i(x,T)$ the index of the second differential of $\S_{\bar\kappa}$ at $(x,T)$. We will also consider the fixed period action functional $\S_{\bar\kappa}^T (x) = \S_{\bar\kappa}(x,T)$. The index of $d^2 \S^T_{\bar\kappa}$ at a critical point $x$ will be denoted by $i_T(x)$. Clearly,
\begin{equation}
\label{conf}
0 \leq i(x,T) - i_T(x) \leq 1.
\end{equation}
When $(x,T)$ is {\em strongly non-degenerate}, meaning that the eigenvalue 1 of $d \phi_{T}(z(0))$ has algebraic multiplicity two and geometric multiplicity one, or equivalently that $P$ does not have the eigenvalue 1 and $T'(\bar\kappa)\neq 0$, the Morse indices of the free period action and of the fixed period action are related by the formula
\begin{equation}
\label{mp}
i(x,T)=\left\{ \begin{array}{cl} i_T(x) + 1  & \mbox{if } T'(\bar\kappa) > 0, \\ i_T(x) & \mbox{it } T'(\bar\kappa)<0. \end{array} \right.
\end{equation}
which is proved in \cite[Theorem 1.3]{mp11}. When $(x,T)$ is transversally non-degenerate but $T'(\bar\kappa)=0$, the second differential $d^2 \S^T_{\bar\kappa}$ of the fixed period action functional has a two-dimensional kernel. The lower-semicontinuity of the Morse index allows to generalize the above formula to the transversally non-degenerate case:

\begin{Proposition} 
\label{full}
Let $(x,T)$ be a transversally non-degenerate critical point of $\S_{\bar\kappa}$, and let 
\[
z_k : \R/T(\kappa)\Z \rightarrow M, \qquad \kappa\in (\bar\kappa-\epsilon, \bar\kappa+\epsilon),
\]
be the corresponding orbit cylinder. Then the Morse indices of the free period and of the fixed period action functional are related by the formula
\[
i(x,T) = \left\{ \begin{array}{cl} i_T(x) + 1  & \mbox{if } T'(\bar\kappa) \geq 0, \\ i_T(x) & \mbox{if } T'(\bar\kappa)<0. \end{array} \right.
\]
\end{Proposition}

\begin{proof}
We just have to deal with the case $T'(\bar\kappa)=0$.
We modify the Hamiltonian $H$ as follows
\[
H_{\lambda} = h_{\lambda} \circ H, \quad \lambda\in \R,
\]
where $h_{\lambda}$ is a diffeomorphism of $\R$ which is $C^2$-close to the identity for $|\lambda|$ small, and such that
\[
h_{\lambda}(\kappa) = \kappa+ \frac{1}{2} \lambda (\kappa - \bar\kappa)^2,
\]
for $\kappa$ close to $\bar\kappa$. In particular,
\[
h_{\lambda}(\bar\kappa)= \bar\kappa, \qquad h_{\lambda}'(\bar\kappa)=(h_{\lambda}^{-1})'(\bar\kappa)=1, \qquad  (h_{\lambda}^{-1})''(\bar\kappa)= - \lambda.
\]
The Hamiltonian $H_{\lambda}$ is Tonelli for $|\lambda|$ small, and we denote by $L_{\lambda}$ its dual Tonelli Lagrangian.
The flow of the vector field $X_{H_{\lambda}}= h_{\lambda}'(H) X_H$ is a time reparametrization of the flow of $X_H$: if $z$ is an orbit of $X_H$ of energy $H=\kappa$, 
\[
t\mapsto z(h_{\lambda}'(\kappa) t) 
\]
is an orbit of $X_{H_{\lambda}}$ of energy $H_{\lambda}=h_{\lambda}(\kappa)$. Since $h_{\lambda}(\bar\kappa) = \bar\kappa$ and $h_{\lambda}'(\bar\kappa)=1$, $z_{\bar\kappa}$ is a closed orbit of $H_{\lambda}$ of energy $H_{\lambda}=\bar\kappa$, and
\[
z_{\kappa}^{\lambda} (t) := z_{h_{\lambda}^{-1}(\kappa)} \bigl( h_{\lambda}'(h_{\lambda}^{-1}(\kappa) )t \bigr) 
\]
is the orbit cylinder of $z_{\bar\kappa}$ with respect to $X_{H_{\lambda}}$. The closed orbit $z_{\kappa}^{\lambda}$ has energy $H_{\lambda}=\kappa$ and period
\[
T_{\lambda}(\kappa) = \frac{T(\kappa)}{h_{\lambda}'(h_{\lambda}^{-1}(\kappa))} = (h_{\lambda}^{-1})'(\kappa) T(\kappa).
\]
Since $T'(\bar\kappa)=0$, we find
\begin{equation}
\label{segno}
T_{\lambda}'(\bar\kappa) = (h_{\lambda}^{-1})''(\bar\kappa) T(\bar\kappa) = - \lambda T(\bar\kappa).
\end{equation}
The point $(x,T)$ is critical for the free-period action functionals which is associated to the Lagrangian $L_{\lambda}$ and the energy $\bar\kappa$, for every $|\lambda|$ small. Being transversally non-degenerate with respect to the Lagrangian $L_0=L$, $(x,T)$ remains transversally non-degenerate for the Lagrangian $L_{\lambda}$ for $|\lambda|$ small enough, and the Morse index $i_{\lambda}(x,T)$ with respect to the free period action functional associated to $L_{\lambda}$ is constant:
\[
i_{\lambda}(x,T) = i(x,T) \qquad \mbox{for $|\lambda|$ small.}
\]
By (\ref{segno}), $(x,T)$ is strongly non-degenerate for $L_{\lambda}$ when $\lambda\neq 0$, and by (\ref{mp}) the index $i_T^{\lambda}(x)$ of $x$ with respect to the fixed period-$T$ action functional associated to $L_{\lambda}$ is
\[
i_T^{\lambda}(x) = \left\{  \begin{array}{cl} i(x,T) - 1  & \mbox{if } \lambda < 0, \\ i(x,T) & \mbox{if } \lambda>0. \end{array} \right.
\]
Since the second differential at $x$ of the fixed period-$T$ action functional associated to $L_{\lambda}$ varies continuously with $\lambda$, the lower semi-continuity of the Morse index implies that
\[
i_T(x) = i_T^{0}(x) = i(x,T) - 1.
\]
\end{proof}

The  mean index with respect to the free period action functional is defined as
$$ \widehat \imath(x,T) := \lim_{n\to\infty} \frac{1}{n} i(x^n,nT), $$
where $x^n(s) := x(ns)$ for every $s\in \T$. By  (\ref{conf}), the above limit exists, as it coincides with the classical mean index with respect to the fixed period action functional
\[
\widehat{\imath}_T (x) := \lim_{n\to\infty} \frac{1}{n} i_{nT}(x^n),
\]
and it is positive if and only if there are conjugate points along the orbit.

\begin{Theorem}
\label{index}
Let $\gamma=(x,T)$ be a transversally non-degenerate critical point of $\S_{\bar\kappa}$, and let $z$ be the corresponding $T$-periodic orbit of $X_H$. Assume that the period $T(\kappa)$ in the orbit cylinder which passes through $z$ satisfies $T'(\bar\kappa)\geq 0$. Then $\widehat \imath(x,T) > 0$. 
As a consequence, if a transversally non-degenerate critical point $(x,T)$ satisfies $i(x,T) \geq 1$,  then $\widehat \imath(x,T) > 0$. 
\end{Theorem}

A prototypical example is the pendulum, with phase space $T^* S^1$: the contractible periodic orbits (i.e. the orbits whose energy is below that of the separatrix) form an orbit cylinder for which $T'> 0$, and indeed these orbits have positive mean index; the non-contractible periodic orbits form an orbit cylinder for which $T'<0$, and they all have zero mean index. 

\begin{proof}
Since $M$ is orientable, the vector bundle $\gamma^*(TM)$ can be trivialized, and the linearization of the Euler-Lagrange equation along $x$ produces a linear second order Lagrangian system in $\R^n$. 
Following \cite{bot56} one can associate to such a linear system an index function $\Lambda: S^1 \to \N$ such that $\Lambda(1) = i_T(x)$. We will make use of the following properties of $\Lambda$. For the proofs we refer to \cite{bot56}, \cite{lon02}, or \cite[Section 2.2]{maz11}.
\begin{enumerate}
\item The discontinuity points of $\Lambda$ are contained in $S^1 \cap \sigma(d\phi_T(z(0)))$, where $\sigma(\cdot)$ indicates the spectrum of a linear map.
\item The splitting numbers $S^\pm(z) := \lim_{\ep\to 0^\pm} \Lambda(e^{i\ep}z) - \Lambda(z)$ are non-negative and depend only on the restriction of $d\phi_T(z(0))$ to the generalised eigenspace associated to the eigenvalue $z$.
\item $i_{nT}(x^n) = \sum_{z^n=1} \Lambda(z)$.
\end{enumerate}
It follows from (iii) that
$$ \widehat \imath(x,T) =\widehat \imath_{T}(x)= \frac{1}{2\pi}\int_{S^1} \Lambda(z)\,dz. $$
Thus, we conclude from the properties above that $\widehat \imath(x,T)=0$ if and only if $i_{nT}(x^n)=0$ for every $n$. 

By property (ii), by the expression (\ref{diffe}) for $d\Phi_T(z(0))$, and by the assumption that $P$ does not have the eigenvalue 1, the splitting numbers $S^\pm(1)$ are determined by the matrix 
\[
\left( \begin{array}{cc} 1 & T'(\bar\kappa) \\ 0 & 1 \end{array} \right).
\]
It turns out that the condition $T^\prime(\bar\kappa)\geq 0$ ensures that $S^+(1)=S^-(1)=1$, see \cite[Examples I and II, page 181]{bot56} and \cite[page 198]{lon02}. Therefore, by properties (i) and (ii) we deduce that $\Lambda(e^{\pm i\ep}) > 0$ if $\ep>0$ is sufficiently small and consequently $\widehat \imath(x,T) > 0$.

Finally, assume that $i(x,T)\geq 1$. If $T'(\bar\kappa)<0$ then the orbit is strongly non-degenerate and (\ref{mp}) implies that
\[
i_T(x) = i(x,T) \geq 1,
\]
so $\widehat{\imath}(x,T)>0$.  If $T'(\bar\kappa)\geq 0$ then $\widehat{\imath}(x,T)>0$ by the above statement.
\end{proof}

The tools in the proof of the previous theorem can also be used to prove the following proposition. Let us recall that a periodic orbit is {\it hyperbolic} if every eigenvalue of its linearized Poincar\'e map $P$ has modulus different from one. We say that a periodic orbit is {\it odd hyperbolic} if it is hyperbolic and the number of eigenvalues (counted with algebraic multiplicity) in the interval $(-1,0)$ is odd. 

\begin{Proposition}
\label{elliptic-odd_hyp}
Let $\gamma=(x,T)$ be a critical point of $\S_{\bar\kappa}$ which is either transversally degenerate or has an odd Morse index $i(x,T)$. Then $\gamma$ is either non-hyperbolic or odd hyperbolic.
\end{Proposition}

\begin{proof}
If $\gamma$ is transversally degenerate then the linearized Poincar\'e map $P$ has the eigenvalue 1, so $\gamma$ is obviously non-hyperbolic. So assume that $\gamma$ is transversally non-degenerate and let $2k+1$, $k\in \N$, be its Morse index. Suppose that $\gamma$ is hyperbolic. We will show that it has to be odd hyperbolic. Since the only eigenvalue of $d\phi_T(z(0))$ in the unit circle is 1, Bott's index function $\Lambda$ is completely determined by the Morse index of $x$ (for the fixed period action functional) and the splitting number $S^+(1)=S^-(1)$. There are two possibilities:

\begin{enumerate}
\item $T^\prime(\bar\kappa) < 0$: in this case $i_T(x)=i(x,T)=2k+1$, by (\ref{mp}), and $S^+(1)=S^-(1)=0$ (see again \cite[Example II, page 181]{bot56}); hence Bott's index function is constantly equal to $2k+1$.
\item $T^\prime(\kappa) \geq 0$: in this case $i_T(x)=i(x,T)-1=2k$, by Proposition \ref{full}, and $S^+(1)=S^-(1)=1$ (see again \cite[Examples I and II, page 181]{bot56}); consequently $\Lambda(1)=2k$ and $\Lambda(z)=2k+1$ for every $z \in S^1\setminus\{1\}$.
\end{enumerate}

By Bott's formula
\[
i_{nT}(x^n) = \sum_{z^n=1} \Lambda(z),
\] 
we conclude that in both cases the indices of the odd and even iterates of $x$ have different parities (more precisely, $i_{nT}(x^n)=(2k+1)n$ in case (i) and $i_{nT}(x^n)=(2k+1)n-1$ in case (ii)).

Following \cite{lon99, lon02}, one can define the index of a symplectic path (starting at the identity) as the infimum of the Conley-Zehnder indices of a non-degenerate perturbation in the $C^0$-topology (see \cite{lon99, lon02} for details). In this way, one can define the index of a periodic orbit as the index of the symplectic path given by the linearized Hamiltonian flow using a symplectic trivialization of $TT^*M$ along the orbit. This index depends on the choice of the trivialization, but its parity does not. Moreover, it coincides with the Morse index for Tonelli Hamiltonians if one uses a {\it vertical} trivialization, that is, a trivialization of $TT^*M$ that sends the vertical distribution to a fixed Lagrangian subspace in $\R^{2n}$, see \cite[Theorem 7.3.4]{lon02} and \cite[Theorem 2.3.5]{maz11}.

Now, write $z^*(TT^*M)$ as a direct sum $\mathcal{V} \oplus \mathcal{W}$, where the subbundles $\mathcal{V}$ and $\mathcal{W}$ are obtained by applying the differential of the flow $\phi_t$ to $V$ and $W$. By construction, these subbundles are invariant by $d\phi_t$, symplectic and symplectic orthogonal. Fix a trivialization of $z^*(TT^*M)$ that sends $\mathcal{V}$ and $\mathcal{W}$ to fixed symplectic subspaces of $\R^{2n}$. Assume also that it sends $\zeta_t := - \frac{\partial}{\partial \kappa} z_{\kappa}(t) \Big|_{\kappa=\bar\kappa}$ and $X_H(z(t))$ to fixed vectors that do not depend on $t$. Notice that this trivialization does not need to be vertical, but, as mentioned above, it does not affect the parity of the index. Consequently, the parities of the indices of the odd and even iterates are different.

By properties of the index, the index of $z$ (with respect to this fixed trivialization) is given by the sum of the indices of the linearized flow restricted to $\mathcal{V}$ and $\mathcal{W}$. The linearized flow restricted to $\mathcal{V}$ is a symplectic shear and it is not hard to see that the parity of its index is the same for every iterate (see \cite[Theorem 8.1.4]{lon02}; in fact, one can actually show that the index is constantly equal to $0$ in case (i) and $-1$ in case (ii) for every iterate).  Thus, the parities of the indices of the even and odd iterates of the linearized flow restricted to $\mathcal{W}$ must be different. But it is well known that this property holds if and only if the number of eigenvalues (counted with algebraic multiplicity) of the linearized return map in the interval $(-1,0)$ is odd (see, for instance, \cite[Proposition 1.4.5]{abb01} or \cite[Lemma 3.2.4]{ust99}).
\end{proof}

\section{Existence of local minimizers}\label{sec:prelim}

In this section we recall some of the main results in \cite{cmp04} and we show how these results imply the existence of local minimizers of the functional $\S_{\kappa}$ on $H^1(\T,M) \times (0,+\infty)$, for every positive energy $\kappa$ below the Ma\~n\'e critical value $c_0$.

Let $M$ be a closed oriented surface and let $p:M_{0}\to M$ denote the abelian cover.
If $\kappa<c_0$, there exists an absolutely continuous closed curve $\gamma:[0,T]\to M_{0}$
with negative ${\S_{\kappa}}$-action (here the action $\S_{\kappa}$ of a closed curve in $M_0$ is associated to the lift of $L$ to $M_0$). By the procedure explained in
\cite[Lemma 3.3]{cmp04} there exists a {\it simple} piecewise smooth closed curve $\beta$  in $M_0$ with constant speed $\sqrt{2k}$ and
negative ${\S_{\kappa}}$-action.  The curve $p\circ\beta$ is homologous to zero
and has negative ${\S_{\kappa}}$-action, but it may not be simple. To remedy this
we pass to a finite cover $N\to M$ as follows. Since the group of covering transformations of $p:M_0\to M$ is abelian, it is residually finite (i.e. the intersection of all its normal subgroups of finite index is trivial). Hence given a compact set $K\subset M_0$ there exist coverings $p_{N}:N\to M$ and $\pi_{N}:M_0\to N$ such that $p_N\circ\pi_N=p$, $p_N$ is a finite covering
and $\pi_{N}|_{K}$ is injective.  Since the image of $\beta$ is a compact set, we can find coverings as above with $\pi_{N}\circ\beta$ a simple closed curve.

Now if $\kappa<c_0$, Theorem 8.5 in \cite{cmp04} shows that there exists a multicurve $\alpha$ in $N$ homologous to zero and with negative ${\S_{\kappa}}$-action such that each component is a simple closed magnetic geodesic with energy $\kappa$. In addition if
 $\tau$ is any other smooth simple closed multicurve homologous to zero, then
\[\sqrt{2\kappa}\,\ell(\tau)+\int_{\tau}\theta\geq \sqrt{2\kappa}\,\ell(\alpha)+\int_{\alpha}\theta={\S_{\kappa}}(\alpha),\]
where $\ell$ denotes length. Using the elementary estimate
\[{\S_{\kappa}}(\tau)\geq \sqrt{2\kappa}\,\ell(\tau)+\int_{\tau}\theta,\]
we deduce that
\[{\S_{\kappa}}(\tau)\geq {\S_{\kappa}}(\alpha)\]
i.e. $\alpha$ is a global minimizer of the free-period action functional in $N$ among all simple closed multicurves homologous to zero. This implies in particular that each component of $\alpha$ is a local minimizer of the free-period action functional  on the space $C^1(\T,N) \times (0,+\infty)$ (because the set of $C^1$-embeddings of the circle is $C^1$-open). 
Since the $\S_{\kappa}$-action of the multicurve is negative, we find at least one closed curve in $N$ which has negative $\S_{\kappa}$-action and which is a local minimizer in $C^1(\T,N) \times (0,+\infty)$. If we project such a curve to $M$ we find a closed curve $\alpha_{\kappa}$ in $M$ which continues to have negative $\S_{\kappa}$-action and is a local minimizer in $C^1(\T,M) \times (0,+\infty)$. The latter assertion follows from the fact that the projection map
\[
C^1(\T,N) \times (0,+\infty) \rightarrow C^1(\T,M) \times (0,+\infty)
\]
is open. The curve $\alpha_{\kappa}$ may not be simple, but that is not an issue for us.

With a slight abuse of terminology, we say that $\alpha=(x,T)$ is a {\em strict local minimizer} of $\S_{\kappa}$ in $C^1(\T,M) \times (0,+\infty)$ (resp.\ in $H^1(\T,M) \times (0,+\infty)$) if the $\T$-orbit of $\alpha$
\[
\T\cdot \alpha = \set{(x(\sigma+\cdot),T)}{\sigma\in \T}
\]
has a neighborhood $\mathcal{U}$ in $C^1(\T,M) \times (0,+\infty)$ (resp.\ in $H^1(\T,M) \times (0,+\infty)$) such that
\[
\S_{\kappa}(\gamma) > \S_{\kappa}(\alpha) \qquad \forall \gamma\in \mathcal{U}\setminus \T\cdot \alpha.
\]
We notice that if $\alpha$ is a local minimizer of $\S_{\kappa}$ but not a strict local minimizer, then there is a sequence $(\alpha_n)\subset (\T \cdot \alpha)^c$ of local minimizers which converges to $\alpha$, and in particular there are infinitely many closed magnetic geodesics.

We need to know that (strict) $C^1$-local minimizers are also (strict) $H^1$-local minimizers. This  follows from the lemma below, whose proof was communicated to us by M.~Mazzucchelli (see also \cite{bn93}, \cite{cha94} and \cite{syc08} for similar results in different settings and with different proofs):

\begin{Lemma} 
\label{minimizers}
Let $\alpha=(x,T)$ be a closed magnetic geodesic which is a local minimzer of $\S_{\kappa}$ in $C^1(\T,M) \times (0,+\infty)$. Then $\alpha$ is also a local minimizer of $\S_{\kappa}$ in $H^1(\T,M) \times (0,+\infty)$. If moreover $\alpha$ is a strict local minimizer in $C^1(\T,M) \times (0,+\infty)$, then it is also a strict local minimizer in $H^1(\T,M) \times (0,+\infty)$.
\end{Lemma}

\begin{proof}
Let us prove the first assertion.
We assume that $\alpha$ is not a local minimizer in $H^1(\T,M)\times (0,+\infty)$ and we prove that it is not a local minimizer in $C^1(\T,M) \times (0,+\infty)$ either. We carry out our argument after a few preliminaries.  
Let $\mathcal{U}$ be a neighborhood of $\alpha$ in $H^1(\T,M) \times (0,+\infty)$ such that the elements $(y,S)\in \mathcal{U}$ have uniform bounds
\[
0 < T_0 \leq S \leq T_1 < +\infty \qquad \mbox{and} \qquad \|\dot{y}\|_{L^2(\T)} \leq C.
\]
In particular, $\mathcal{U}$ is an equicontinuous family of periodic curves $\gamma:\R \rightarrow M$. Then we can find a natural number $h$ which is so large that the following holds: for every $\gamma =(y,S) \in \mathcal{U}$ and every $j=0,1,\dots,h-1$ there is a unique curve $[jS/h,(j+1)S/h] \rightarrow M$ which minimizes  the Lagrangian action among all absolutely continuous curves on $[jS/h,(j+1)S/h]$ with end points $\gamma(jS/h)$ and $\gamma((j+1)S/h)$ (see e.g.\ \cite[Theorem 4.1.1]{maz11}). For such a large $h$ we can define the continuous map
\[
\Lambda: \mathcal{U} \rightarrow H^1(\T,M) \times (0,+\infty), \qquad \gamma \mapsto \hat{\gamma},
\]
which maps every $\gamma=(y,S)\in \mathcal{U}$ to the unique curve $\hat{\gamma}:\R/S \Z \rightarrow M$ such that for every $j=0,1,\dots,h-1$:
\begin{enumerate}
\item $\hat{\gamma} ( jS/h) = \gamma (jS/h)$;
\item $\hat{\gamma}|_{[jS/h,(j+1)S/h]}$ minimizes the Lagrangian action among the absolutely continuous curves on $[jS/h,(j+1)S/h]$ with end points $\gamma(jS/h)$ and $\gamma((j+1)S/h)$.
\end{enumerate}
The continuity of $\Lambda$ is a consequence of \cite[Theorem 4.1.2]{maz11} (and holds with a target space having a much finer topology). The curve $\hat{\gamma} =\Lambda(\gamma)$ is a $h$-broken solution of the Euler-Lagrange equation of $L$. By (ii) we have
\[
\S_{\kappa}\bigl(\Lambda(\gamma) \bigr) \leq \S_{\kappa}(\gamma) \qquad \forall \gamma\in \mathcal{U}.
\]
Moreover, being a smooth solution of the Euler-Lagrange equation of $L$ and by our choice of $h$, $\alpha$ is a fixed point of $\Lambda$. 

Let $\rho\in C^{\infty}(\T)$ be a smooth non-negative function with support in $(-1/2,1/2)$ and integral 1. For each $\epsilon>0$ we set $\rho_{\epsilon}(t) = \rho(t/\epsilon)/\epsilon$, so that $\rho_{\epsilon}$ converges to the Dirac delta $\delta_0$ in the sense of distributions for $\epsilon\rightarrow 0$. By Whitney's theorem, there exists an embedding $M\hookrightarrow \R^N$ and, for a tubular neighborhood $U\subset \R^N$ of $M$, a smooth retraction $r:U \rightarrow M$. For each sufficiently small $\epsilon>0$ we define a continuous mapping
\[
\Theta_{\epsilon} : \mathcal{U} \rightarrow C^{\infty}(\T,M) \times (0,+\infty)
\]
by $\Theta_{\epsilon}(y,S):= (y_{\epsilon},S)$, where
\[
y_{\epsilon} := r \circ (y * \rho_{\epsilon} ),
\]
and $*$ denotes the convolution on $\T$.
 
Consider a sequence $\gamma_n = (x_n,T_n)$ which converges to $\alpha=(x,T)$ in $H^1(\T,M) \times (0,+\infty)$ and such that
\[
\S_{\kappa} (\gamma_n) < \S_{\kappa} (\alpha) \qquad \forall n\in \N.
\]
The sequence $\hat{\gamma}_n = (\hat{x}_n,T_n) := \Lambda(\gamma_n)$ also converges to $\alpha$ in $H^1(\T,M) \times (0,+\infty)$ and
\begin{equation}
\label{sotto1}
\S_{\kappa} (\hat{\gamma}_n) \leq \S_{\kappa} (\gamma_n) < \S_{\kappa} (\alpha) \qquad \forall n\in \N.
\end{equation}
Each restriction $\hat{x}_n|_{[j/h,(j+1)/h]}$ is a reparametrized Euler-Lagrange curve whose end-points converge to $x(j/h)$ and $x((j+1)/h)$. By the continuous dependence of the absolute minimizers of the Lagrangian action with respect to the end-points (see \cite[Theorem 4.1.2]{maz11}), we have the convergence
\begin{equation}
\label{conv}
\hat{x}_n|_{[j/h,(j+1)/h]} \rightarrow  x|_{[j/h,(j+1)/h]}
\end{equation}
in the $C^{\infty}$-topology, for every $j=0,1\dots,h-1$.

By (\ref{sotto1}), we can find an infinitesimal sequence $(\epsilon_n) \subset (0,+\infty)$ such that
\[
\S_{\kappa}\bigl( \Theta_{\epsilon_n} (\hat{\gamma}_n) \bigr) < \S_{\kappa} (\alpha)\qquad \forall n\in \N.
\]
There remains to prove that $\Theta_{\epsilon_n} (\hat{x}_n)$ converges to $x$ in the $C^1$-topology.

Since $\hat{x}_n$ is continuous and piece-wise smooth, it is absolutely continuous and its a.e.\ defined pointwise derivative agrees with its distributional derivative. Therefore
\[
\begin{split}
\left\| \frac{d}{ds}  \bigl( \hat{x}_n * \rho_{\epsilon_n} \bigr) - \right. & \left. \frac{d}{ds} x \right\|_{L^{\infty}(\T)}  =
\left\| \Bigl( \frac{d}{ds} \bigl( \hat{x}_n - x \bigr) \Bigr)* \rho_{\epsilon_n}   + \Bigl( \frac{d}{ds} x \Bigr) * ( \rho_{\epsilon_n} - \delta_0 ) \right\|_{L^{\infty}(\T)} \\
&\leq  \left\| \Bigl( \frac{d}{ds} \bigl( \hat{x}_n - x \bigr)  \Bigr)* \rho_{\epsilon_n}  \right\|_{L^{\infty}(\T)} + \left\| \Bigl( \frac{d}{ds} x \Bigr) * ( \rho_{\epsilon_n} - \delta_0 ) \right\|_{L^{\infty}(\T)} \\ &\leq \underbrace{ \left\| \frac{d}{ds} \bigl( \hat{x}_n - x \bigr)  \right\|_{L^{\infty}(\T)}}_{=:p_n} \cdot \underbrace{ \left\| \rho_{\epsilon_n} \right\|_{L^1(\T)}}_{=1} + \underbrace{ \left\| \Bigl( \frac{d}{ds} x \Bigr) * ( \rho_{\epsilon_n} - \delta_0 ) \right\|_{L^{\infty}(\T)}}_{=:q_n}.
\end{split}
\]
The sequence $(p_n)$ is infinitesimal by (\ref{conv}). The sequence $(q_n)$ is also infinitesimal because $x$ is smooth and $(\rho_{\epsilon_n})$ converges to $\delta_0$  in the distributional sense. Therefore,
\[
\left\| \frac{d}{ds}  \bigl( \hat{x}_n * \rho_{\epsilon_n} \bigr) - \frac{d}{ds} x\right\|_{L^{\infty}(\T)} \rightarrow 0\qquad \mbox{for } n\rightarrow \infty.
\]
Since the retraction $r$ is smooth and fixes $x$, we conclude that
\[
\left\| \frac{d}{ds}  \Theta_{\epsilon} (\hat{x}_n) - \frac{d}{ds} x \right\|_{L^{\infty}(\T)} \rightarrow 0\qquad \mbox{for } n\rightarrow \infty.
\]
This proves the first statement.

Now assume that the local minimizer $\alpha=(x,T)$ is not strict in $H^1(\T,M) \times (0,+\infty)$. Then we can find a sequence of local minimizers $\gamma_n = (x_n,T_n)$ in the complement of $\T\cdot \alpha$ which converge to $\alpha$ in $H^1(\T,M) \times (0,+\infty)$. In particular, $(x_n)$ converges to $x$ uniformly and, up to a subsequence, $(dx_n/ds)$ converges to $dx/ds$ almost everywhere. Since the curves $\gamma_n$ are solutions of the Euler-Lagrange equation, the smooth dependence of the Cauchy problem on initial data implies that $(x_n)$ converges to $x$ in $C^{\infty}(\T,M)$. In particular, the local minimizer $\alpha=(x,T)$ is not strict in $C^1(\T,M) \times (0,+\infty)$, as we wished to prove.
\end{proof}

Summarizing, we have proved the following:

\begin{Lemma}\label{lemma:consecCMP} 
For every $\kappa\in (0,c_0)$ there is a closed magnetic geodesic $\alpha_{\kappa}$ with energy $\kappa$ and $\S(\alpha_{\kappa})<0$ which is a local minimizer of $\S_{\kappa}$ in $H^1(\T,M) \times (0,+\infty)$.
\end{Lemma}

We observe that in principle $\alpha_{\kappa}$ may belong to a non-trivial homotopy class and might even be not null-homologous (since we have chosen one suitable component of the null-homologous cycle $\alpha$).

\section{Persistence of local minimizers}

Set $\mathcal{M} := H^1(\T,M) \times (0,+\infty)$.
The fact that $M$ is an orientable surface implies that a closed curve in $M$ which is a (strict) local minimizer of $\S_{\kappa}$ on $\mathcal{M}$ remains a (strict) local minimizer also when iterated:

\begin{Lemma}
\label{persistence}
If $\alpha:[0,T] \rightarrow M$ is a local minimizer (resp.\ strict local minimizer) of $\S_{\kappa}$ on $\mathcal{M}$, then for every $n\geq 1$ its $n$-th iterate $\alpha^n(t):[0,nT]\rightarrow M$ is also a local minimizer (resp.\ strict local minimizer) of $\S_{\kappa}$ on $\mathcal{M}$.
\end{Lemma}

This type of result fails in dimensions greater than or equal to three even in the Riemannian case as examples of Hedlund show \cite{H}.
It also fails for non-orientable surfaces, cf. \cite[Example 9.7.1]{KH}.

The proof of Lemma \ref{persistence} hinges on the following lemma, in which we say that a sequence of $C^1$ curves $\gamma_h:[S_h,T_h] \rightarrow M$ converges to a curve $\gamma:[S,T] \rightarrow M$ in $C^1$ iff $T_h-S_h\rightarrow T-S$ and $x_h(s):= \gamma_h(S_h+(T_h-S_h)s)$ converges to $x(s):= \gamma(S+(T-S)s)$ in $C^1([0,1],M)$.

\begin{Lemma}
\label{split}
Let $\gamma:[0,T] \rightarrow M$ be a smooth immersed closed curve, let $n\geq 2$ be an integer, and let $\gamma_h:[0,T_h]\rightarrow M$ be a sequence of $C^1$ closed curves  which converges to the $n$-th iterate $\gamma^n:[0,nT]\rightarrow M$ in $C^1$. Then, up to a subsequence and up to time-shifts in parametrization of the closed curves $\gamma_h$ and $\gamma$, there exists a sequence $S_h\in [0,T_h]$ such that:
\begin{enumerate}
\item $S_h\rightarrow T$;
\item $\gamma_h(0)=\gamma_h(S_h)=\gamma_h(T_h)$;
\item the sequence $\gamma_h|_{[0,S_h]}$ converges to $\gamma$ in $C^1$;
\item the sequence $\gamma_h|_{[S_h,T_h]}$ converges to $\gamma^{n-1}$ in $C^1$.
\end{enumerate}
\end{Lemma}

\begin{proof} Extend the functions $\gamma$ and $\gamma_h$ to the whole $\R$ by periodicity. Since $M$ is an orientable surface, we can choose coordinates $(\theta,\lambda)\in \R/T\Z \times \R$ on a neighborhood of the immersed closed curve $\gamma(\R)$ in such a way that
\[
\gamma(t) = (t + T\Z,0) \qquad \forall t\in \R.
\]
Up to neglecting finitely many terms, the sequence $(\gamma_h)$ consists of curves whose image is in such a neighborhood and hence 
\[
\gamma_h(t) = \bigl( \theta_h(t) + T\Z, \lambda_h(t)\bigr),
\]
where the sequences of $C^1$ functions $\theta_h,\lambda_h: \R \rightarrow \R$ converge to the identity and to the zero function in $C^1(\R,\R)$, $\lambda_h$ is $T_h$-periodic, and $\theta_h$ satisfies
\[
\theta_h(t+T_h) = nT + \theta_h(t)\qquad \forall t\in \R.
\] 
Up to neglecting finitely many terms, $\theta_h:\R \rightarrow \R$ is a $C^1$ diffeomorphism and its inverse satisfies
\begin{equation}
\label{perthe}
\theta_h^{-1}(s+nT) = T_h + \theta_h^{-1}(s) \qquad \forall s\in \R.
\end{equation}
Fix some $h\in \N$. The real bi-infinite sequence
\[ 
(\mu_j)_{j\in \Z} :=  \Bigl( \lambda_h\bigl( \theta_h^{-1}(jT) \bigr) \Bigr)_{j\in \Z}
\]
is $n$-periodic. In particular, it cannot be strictly monotone: w.l.o.g.\ we can find an integer $k=k_h\in [1,n-1]$ such that
\[
\mu_{k-1} \leq \mu_k \quad \mbox{and} \quad \mu_k \geq \mu_{k+1}.
\]
Then the continuous function
\[
f(\sigma):= \lambda_h \Bigl( \theta_h^{-1} \bigl( (k+\sigma) T\bigr) \Bigr) - \lambda_h \Bigl( \theta_h^{-1} \bigl( (k-1+\sigma) T\bigr) \Bigr)
\]
satisfies $f(0)\geq 0$ and $f(1)\leq 0$. Therefore, there exists $\sigma_h\in [0,1]$ such that $f(\sigma_h)=0$. Set
\[
s_h := \theta_h^{-1} \bigl( (k_h-1 + \sigma_h) T) \quad \mbox{and} \quad t_h := \theta_h^{-1} \bigl( (k_h + \sigma_h) T).
\]
Up to a subsequence, $(s_h)$ converges to some $\bar{s}$ and $(t_h)$ converges to some $\bar{t}$ with $\bar{t}-\bar{s}=T$. Since $\sigma_h$ is a zero of $f$, we find
\[
\begin{split}
\gamma_h(s_h) &= \Bigl( (k_h-1+\sigma_h) T + T\Z, \lambda_h \bigl(\theta_h^{-1}((k_h-1+\sigma_h)T)\bigr) \Bigr) \\ &= \Bigl( (k_h+\sigma_h)T + T\Z, \lambda_h \bigl( \theta_h^{-1}((k_h+\sigma_h)T) - T_h\bigr) \Bigr) \\ &= \Bigl( (k_h+\sigma_h)T + T\Z, \lambda_h \bigl( \theta_h^{-1}((k_h+\sigma_h)T) \bigr) \Bigr) = \gamma_h(t_h),
\end{split}
\]
where we have used (\ref{perthe}) and the $T_h$-periodicity of $\lambda_h$. Moreover, the sequence  $\gamma_h|_{[s_h,t_h]}$ converges to  $\gamma|_{[\bar{s},\bar{t}]}$ in $C^1$, while $\gamma_h|_{[t_h,s_h+T_h]}$ converges to $\gamma|_{[\bar{T},\bar{t}+nT]}$ in $C^1$.  The conclusion follows by shifting $\gamma_h$ by $s_h$ and $\gamma$ by $\bar{s}$, and by setting $S_h:= t_h-s_h$.
\end{proof}

\begin{proof}[Proof of Lemma \ref{persistence}.]  We assume that $\alpha$ is a local minimizer (resp.\ strict local minimizer) and we prove that also $\alpha^n$ is a local minimizer (resp.\ strict local minimizer) for every $n\in \N$.
We argue by induction on $n$, the case $n=1$ being trivially true. Assume the statement to be true for $n-1$ and, by contradiction, that it fails for some $n\geq 2$. Then Lemma \ref{minimizers} implies that $\alpha^n$ is not a local minimizer (resp.\ not a strict local minimizer) in the $C^1$ topology: there exists a sequence of closed $C^1$ curves $\gamma_h:[0,T_h]\rightarrow M$  which converges to $\alpha^n$ in $C^1$ and satisfies
\begin{equation}
\label{contr}
\begin{split}
& \S_{\kappa}(\gamma_h)<\S_k(\alpha^n) = n \S_{\kappa} (\alpha) \\  \mbox{(resp. } \gamma_h \notin \T\cdot \alpha^n \mbox{ and } & \S_{\kappa}(\gamma_h)\leq \S_k(\alpha^n) = n \S_{\kappa} (\alpha) \mbox{ )}.
\end{split}
\end{equation}
By Lemma \ref{split}, up to a subsequence and time-shifts, $\gamma_h$ is the juxtaposition of two curves $\gamma_h|_{[0,S_h]}$ and $\gamma_h|_{[S_h,T_h]}$ such that $\gamma_h(0)=\gamma_h(S_h)=\gamma_h(T_h)$, which converge to $\alpha$ and $\alpha^{n-1}$ in $C^1$, respectively. The curves $\gamma_h|_{[0,S_h]}$ and $\gamma_h|_{[S_h,T_h]}$ belong a fortiori to $\mathcal{M}$, and they converge to $\gamma$ and $\gamma^n$ in the topology of $\mathcal{M}$. 
Since $\alpha$ and $\alpha^{n-1}$ are local minimizers (resp.\ strict local minimizers) in $\mathcal{M}$ by the inductive hypotesis,
\[
\begin{split}
& \S_{\kappa}\bigl( \gamma_h|_{[0,S_h]} \bigr) \geq \S_{\kappa}(\gamma), \quad 
\S_{\kappa}\bigl( \gamma_h|_{[S_h,T_h]} \bigr) \geq \S_{\kappa}(\gamma^{n-1}) = (n-1) \S_{\kappa}(\gamma), \\
\mbox{(resp. } & \S_{\kappa}\bigl( \gamma_h|_{[0,S_h]} \bigr) > \S_{\kappa}(\gamma) \mbox{ or } \gamma_h|_{[0,S_h]} \in \T\cdot \alpha, \\
& \S_{\kappa}\bigl( \gamma_h|_{[S_h,T_h]} \bigr) > \S_{\kappa}(\gamma^{n-1}) = (n-1) \S_{\kappa}(\gamma) \mbox{ or } \gamma_h|_{[S_h,T_h]} \in \T\cdot \alpha^{n-1} \mbox{ )},
\end{split}
\]
for $h$ large enough, from which we obtain
\[
\begin{split}
& \S_{\kappa}(\gamma_h) = \S_{\kappa}\bigl( \gamma_h|_{[0,S_h]}  \bigr) + \S_{\kappa}\bigl( \gamma_h|_{[S_h,T_h]}\bigr) \geq  \S_{\kappa}(\gamma) + (n-1) \S_{\kappa}(\gamma) = n \S_{\kappa} (\gamma), \\
\mbox{(resp. } & \S_{\kappa}(\gamma_h) >  \S_{\kappa}(\gamma) + (n-1) \S_{\kappa}(\gamma) = n \S_{\kappa} (\gamma) \mbox{ or } \gamma_h \in \T\cdot \alpha^n \mbox{ )},
\end{split}
\]
which contradicts (\ref{contr}).
\end{proof}

\section{The negative gradient flow of $\S_{\kappa}$ and strict local minimizers}

Let $\varphi:(0,+\infty) \rightarrow \R$ be a smooth function such that $\varphi(T)=T^2$ for $T\leq 1/2$ and $\varphi(T)=1$ for $T\geq 1$. Following \cite{con06}, we endow the Hilbert manifold $\mathcal{M}= H^1(\T,M)\times (0,+\infty)$ with the Riemannian structure
\[
\langle (\xi_1,\tau_1), (\xi_2,\tau_2) \rangle_{(x,T)} := \tau_1 \tau_2 + \varphi(T) 
 \langle \xi_1,\xi_2 \rangle_x,
\]
where $(x,T)\in \mathcal{M}$,
\[
(\xi_1,\tau_1), (\xi_2,\tau_2) \in T_x \mathcal{M} = T_x H^1(\T,M) \times \R,
\]
and $\langle \cdot,\cdot \rangle_x$ denotes the standard Riemannian structure on $H^1(\T,M)$ which is induced by a Riemannian structure on $M$. Since $\varphi(T) \rightarrow 0$ for $T\rightarrow 0$, this metric has more non-converging Cauchy sequences than the product one and is a fortiori not complete. 

The functional $\S_{\kappa}$ is smooth on $\mathcal{M}$, and 
$\nabla \S_{\kappa}$ denotes its gradient vector field with respect to the Riemannian structure defined above. We shall use the following results from \cite{con06}.

\begin{Lemma}[\cite{con06}, Lemma 6.9]
\label{gon1}
Let $(x,T):[0,\rho) \rightarrow \mathcal{M}$ be a flow line of $-\nabla \S_{\kappa}$ such that 
\[
\liminf_{r\rightarrow \rho} T(r)=0.
\]
Then 
\[
\lim_{r\rightarrow \rho} \S_{\kappa}\bigl(x(r),T(r)\bigr) = 0.
\]
\end{Lemma}

\begin{Lemma}
\label{gon2}
Let $\kappa>0$ and let $(x_h,T_h)$ be a Palais-Smale sequence for $\S_{\kappa}$. If $(T_h)$ is bounded, then $(x_h,T_h)$ has a convergent subsequence in $\mathcal{M}$.
\end{Lemma}

\begin{proof}
The sequence $(T_h)$ is bounded away from zero: if not, Proposition 3.8 in \cite{con06} would imply that a subsequence of $(x_h)$ converges to a constant loop which is an equilibrium orbit with energy $\kappa$. But in the case of a magnetic Lagrangian, all constant loops are equilibrium orbits with zero energy.
The results now follows from the fact that
Palais-Smale sequences $(x_h,T_h)$ with $(T_h)$ bounded and bounded away from zero have a converging subsequence by \cite[Proposition 3.12]{con06}.
\end{proof}

The above result implies in particular that the Palais-Smale condition holds locally, and this allows to prove that the $\T$-orbit of a strict local minimizer $\alpha$ has neighborhoods on whose boundary the infimum of $\S_{\kappa}$ is strictly larger than $\S_{\kappa}(\alpha)$:

\begin{Lemma}  
\label{strict}
Let $\alpha=(x,T)$ be a strict local minimizer of $\S_{\kappa}$ on $\mathcal{M}$. If the neighborhood $\mathcal{U}$ of $\T\cdot \alpha$ is sufficiently small, then
\[
\inf_{\partial \mathcal{U}} \S_{\kappa} > \S_{\kappa}(\alpha).
\]
\end{Lemma}

\begin{proof}
Denote by $\mathcal{N}_{\delta}$ the open $\delta$-neighborhood of the set $\T \cdot \alpha$ in $\mathcal{M}$. Let $\delta>0$ be so small that the gradient of $\S_{\kappa}$ is bounded on $\overline{\mathcal{N}_{\delta}}$ and there holds
\begin{equation}
\label{deltasmall}
\S_{\kappa}(\gamma) > \S_{\kappa}(\alpha) \qquad \forall \gamma \in \overline{\mathcal{N}_{\delta}} \setminus \T \cdot \alpha.
\end{equation}
Let $\mathcal{U}$ be a neighborhood of $\T \cdot \alpha$ which is contained in $\mathcal{N}_{\delta/2}$, and assume by contradiction that
\[
\inf_{\partial \mathcal{U}} \S_{\kappa} = \S_{\kappa}(\alpha).
\]
Let $\epsilon>0$ be such that
\[
\partial \mathcal{U} \subset \overline{N_{\delta/2}} \setminus \mathcal{N}_{\epsilon}.
\]
Let $\phi$ be the negative gradient flow of $\S_{\kappa}$. Since the gradient of $\S_{\kappa}$ is bounded on $\overline{\mathcal{N}_{\delta}}$, the above inclusion implies the existence of a positive number $\rho$ such that
\[
\phi \bigl( [0,\rho] \times \partial \mathcal{U} \bigr) \subset \overline{\mathcal{N}_{\delta}} \setminus \mathcal{N}_{\epsilon/2}.
\]
Let $(\gamma_h)\subset \partial \mathcal{U}$ be a sequence such that $(\S_{\kappa}(\gamma_h))$ converges to $\S_{\kappa}(\alpha)$. 
Since 
\[
\begin{split}
\min_{r\in [0,\rho]} \bigl\| d\S_{\kappa} (\phi_r(\gamma_h))\bigr\|^2 &\leq \frac{1}{\rho} \int_0^{\rho} \bigl\| d\S_{\kappa} (\phi_r(\gamma_h))\bigr\|^2\, dr = - \frac{1}{\rho} \int_0^{\rho} \frac{d}{dr} \S_{\kappa} (\phi_r(\gamma_h))\, dr \\ &= \frac{1}{\rho} \bigl(
\S_{\kappa}(\gamma_h) - \S_{\kappa}(\phi_{\rho}(\gamma_h)) \bigr) \leq \frac{1}{\rho} \bigl(
\S_{\kappa}(\gamma_h) - \S_{\kappa}(\alpha)\bigr) \rightarrow 0,
\end{split}
\]
the functional $\S_{\kappa}$ has a Palais-Smale sequence 
\[
\tilde{\gamma}_h = \phi(\rho_h,\gamma_h) \in \overline{\mathcal{N}_{\delta}} \setminus \mathcal{N}_{\epsilon/2}, \qquad \rho_h \in [0,\rho],
\]
such that $(\S_{\kappa}(\tilde{\gamma}_h))$ converges to $\S_{\kappa}(\alpha)$. By Lemma \ref{gon2} we deduce the existence of a critical point 
\[
\gamma\in \overline{\mathcal{N}_{\delta}} \setminus \mathcal{N}_{\epsilon/2},
\]
such that $\S_{\kappa}(\gamma) = \S_{\kappa}(\alpha)$. This contradicts (\ref{deltasmall}).
\end{proof}

\section{The minimax values} 

For any $\kappa\in (0,c_u)\subset (0,c_0)$ let $\alpha_{\kappa}\in \mathcal{M}$ be a local minimizer of $\S_{\kappa}$ with $\S_{\kappa}(\alpha_{\kappa})<0$, whose existence is guaranteed by Lemma \ref{lemma:consecCMP}. 
Let $P\subset (0,c_u)$ be the set of values of $\kappa$ for which the local minimizer $\alpha_{\kappa}$ is strict, and let $Q\subset P$ be the set of values $\kappa$ for which $\alpha_{\kappa}$ is transversally non-degenerate (see the Introduction and Section \ref{Sec:meanindex} for the definition). By Lemma \ref{persistence}, $\alpha_{\kappa}^n$ is a strict local minimizer for every $\kappa\in P$ and every $n\in \N$ (when $\kappa\in Q$, $\alpha_{\kappa}^n$ may not be transversally non-degenerate for $n\geq 1$, because the linearized Poincar\'e map of $\alpha_{\kappa}$ may have eigenvalues which are roots of 1).

We shall define two minimax values: the first one is associated to energies close to a value $\kappa_*\in P$, the second one is associated to energies close to a value $\kappa_*\in Q$.

We begin with the case of an energy level $\kappa_*$ in $P$. Since $\kappa_*< c_u$, the infimum of $\S_{\kappa_*}$ over all contractible curves is $-\infty$, so we can find an element $\mu \in \mathcal{M}$ in the same free homotopy class of $\alpha_{\kappa_*}$  such that
\[
\S_{\kappa_*} (\mu) < \S_{\kappa_*}(\alpha_{\kappa_*})<0.
\]
Since $\S_{\kappa} \rightarrow \S_{\kappa_*}$ pointwise for $\kappa \rightarrow \kappa_*$, we can find an open interval $I \subset (0,c_u)$ containing $\kappa_*$ such that
\begin{equation}
\label{s2}
\S_{\kappa}(\mu) <  \S_{\kappa} (\alpha_{\kappa_*})<0, \qquad \forall \kappa\in I.
\end{equation}
For every $n\in \N$ we set
\[
\mathcal{P}_n := \set{u\in C^0([0,1], \mathcal{M}) }{ u(0) = \alpha_{\kappa_*}^n, \; u(1) = \mu^n},
\]
and we define a function $p_n : I \rightarrow \R$ by
\[
p_n(\kappa) := \inf_{u\in \mathcal{P}_n} \max_{\sigma\in [0,1]} \S_{\kappa}\bigl(u(\sigma)\bigr).
\]
Since $\S_{\kappa}$ depends monotonically on $\kappa$, the function $\kappa \mapsto p_n(\kappa)$ is (not necessarily strictly) increasing for every $n\in \N$. 

Now let $\kappa_*\in Q$. Therefore, the kernel of $d^2 \S_{\kappa_*} (\alpha_{\kappa_*})$ is one-dimensional and coincides with the tangent line to $\T\cdot \alpha_{\kappa_*}$ at $\alpha_{\kappa_*}$. Let $I\subset  (0,c_u)$ be an open interval containing $\kappa_*$ for which there is a smooth orbit cylinder $\{\tilde{\alpha}_{\kappa}\}_{\kappa\in I}$ with $\tilde{\alpha}_{\kappa_*} = \alpha_{\kappa_*}$: $\tilde{\alpha}_{\kappa}$ is a $T(\kappa)$-periodic orbit of energy $\kappa$, with $T'(\kappa_*) \neq 0$. The closed curve $\tilde{\alpha}_{\kappa}$ may or may not coincide with $\alpha_{\kappa}$.
Up to reducing $I$, we may assume that $\tilde{\alpha}_{\kappa}$ is a transversally non-degenerate local minimizer of $\S_{\kappa}$ for every $\kappa\in I$. 

Since $\kappa_*<c_u$ we can find a closed curve $\mu\in \mathcal{M}$ in the same free homotopy class of $\alpha_{\kappa_*}$ such that
\[
\S_{\kappa_*} (\mu) < \S_{\kappa_*} (\alpha_{\kappa_*})<0.
\]
Up to reducing $I$ even more, we may assume that
\begin{equation}
\label{sotto}
\S_{\kappa} (\mu) < \S_{\kappa} (\tilde{\alpha}_{\kappa})< 0 \qquad \forall \kappa\in I.
\end{equation}

For every $\kappa\in I$ and every $n\in \N$ we set
\[
\mathcal{Q}_n (\kappa) := \set{u\in C^0([0,1],\mathcal{M})}{u(0)=\tilde{\alpha}_{\kappa}^n, \; u(1) = \mu^n},
\]
and we define a function $q_n: I \rightarrow \R$ by
\[
q_n(\kappa) := \inf_{u\in \mathcal{Q}_n(\kappa)} \max_{\sigma\in [0,1]} \S_{\kappa}\bigl(u(\sigma)\bigr) .
\]
Notice that, unlike $\mathcal{P}_n$, the class $\mathcal{Q}_n(\kappa)$ depends on the energy level $\kappa$. Therefore, the monotonicity of $q_n$ does not follow directly from the monotonicity of $\kappa \mapsto \S_{\kappa}$, but requires a proof:

\begin{Lemma}
The function $q_n: I \rightarrow \R$ is monotonically increasing.
\end{Lemma}

\begin{proof}
Let $\kappa_0 < \kappa_1$ be elements of $I$. For every $u\in \mathcal{Q}_n(\kappa_1)$ we define the path $v\in \mathcal{Q}_n(\kappa_0)$ as
\[
v(\sigma) := \left\{ \begin{array}{ll} \tilde{\alpha}^n_{\kappa_0 + 2\sigma (\kappa_1 - \kappa_0)} , & \mbox{if } 0\leq \sigma \leq 1/2, \\ u(2\sigma -1), & \mbox{if } 1/2 < \sigma \leq 1. \end{array} \right.
\]
Since $\tilde{\alpha}_{\kappa}$ is a critical point of $\S_{\kappa}$, we find
\[
\frac{d}{d\kappa} \S_{\kappa} (\tilde{\alpha}_{\kappa}) = d \S_{\kappa} (\tilde{\alpha}_{\kappa}) \Bigl[ \frac{\partial \tilde{\alpha}_{\kappa}}{\partial \kappa}  \Bigr] + \frac{\partial \S_{\kappa}}{\partial \kappa} (\tilde{\alpha}_{\kappa}) = \frac{\partial \S_{\kappa}}{\partial \kappa} (\tilde{\alpha}_{\kappa}) = T(\kappa)>0.
\]
Therefore we have for $\kappa\in [\kappa_0,\kappa_1]$
\[
\S_{\kappa_0} (\tilde{\alpha}_{\kappa}^n) = n
\S_{\kappa_0} (\tilde{\alpha}_{\kappa}) \leq n \S_{\kappa} (\tilde{\alpha}_{\kappa}) \leq n \S_{\kappa_1} (\tilde{\alpha}_{\kappa_1}) = \S_{\kappa_1} (\tilde{\alpha}_{\kappa_1}^n) = \S_{\kappa_1}(u(0)),
\]
from which we obtain
\[
\max_{\sigma \in [0,1]} \S_{\kappa_0} (v(\sigma)) \leq \max_{\sigma \in [0,1]} \S_{\kappa_1} (u(\sigma)).
\]
By taking the infimum over all $u\in \mathcal{Q}_n(\kappa_1)$ we conclude that $q_n(\kappa_0) \leq q_n(\kappa_1)$.
\end{proof}

The following lemma is based on an argument which is due to V.\ Bangert \cite{ban80}:

\begin{Lemma}
\label{bangert}
Let $\mu_0,\mu_1\in \mathcal{M}$ be in the same free homotopy class, and let
\[
\mathcal{R}_n := \set{u\in C^0([0,1], \mathcal{M})}{u(0) = \mu_0^n, \; u(1)=\mu_1^n}.
\]
Fix a number $\kappa$ and set
\[
c_n := \inf_{u\in \mathcal{R}_n} \max_{\sigma\in [0,1]} \S_{\kappa}(u(\sigma)).
\]
Then there exists a number $A$ such that
\[
c_n \leq n \, \max\{\S_{\kappa}(\mu_0), \S_{\kappa}(\mu_1)\} + A\qquad \forall n\in \N.
\]
\end{Lemma}

\begin{proof}
Let $u=(x,T)\in \mathcal{R}_1$ be such that the curve 
\[
\tau: [0,1] \rightarrow M, \qquad \tau(\sigma):= x(\sigma)(0)
\]
is smooth and let $\hat\tau$ be the inverse curve, $\hat{\tau}(\sigma) := \tau(1-\sigma)$.

Let $n$ be a natural number. Let $v\in \mathcal{R}_n$ be the homotopy connecting $\mu_0^n$ to $\mu_1^n$ which is obtained from $u$ by pulling one loop at a time (see \cite[Fig.\ 1]{ban80}). That is, for a typical value of $\sigma$ in $[0,1]$, $v(\sigma)$ is obtained as the juxtaposition of the following curves:
\begin{equation}
\label{center}
\begin{split}
\mu_0^h &: [0,hT_0] \rightarrow M, \\
\tau|_{[0,s]} &: [0,s] \rightarrow M, \\
\gamma_s &: [0,T_s] \rightarrow M, \\
\tau|_{[s,1]} &: [s,1] \rightarrow M, \\
\mu_1^{n-h-1} &: [0,(n-h-1)T_1] \rightarrow M, \\
\hat\tau &:[0,1] \rightarrow M,
\end{split}
\end{equation}
for some natural number $h=h(\sigma)\leq n-1$ and for some real number $s=s(\sigma) \in [0,1]$. For other values of $\sigma\in [0,1]$ the form of the curve is different, because the first loop $\mu_0$ has still to be transported to $\mu_1$, or all the loops have been transported to $\mu_1^n$ and there is still to eliminate the curves $\tau$ and $\hat\tau$. See \cite[proof of Theorem 1]{bk83} for the precise construction. 

Notice that here we never reparametrize the curves, we just juxtapose them by shifting their original parametrization. The free period action functional $\S_{\kappa}$ is additive with respect to such a way of juxtaposing curves, therefore the action of the curve $v(\sigma)$ defined by (\ref{center}) is
\[
\S_{\kappa} \bigl( v(\sigma) \bigr) = h \S_{\kappa} (\mu_0) + \S_{\kappa} (\tau|_{[0,s]}) + \S_{\kappa} (\gamma_s) + (n-h-1) \S_{\kappa} (\mu_1) + \S_{\kappa} ( \hat\tau ).
\] 
This quantity can be bounded from above by
\[
\S_{\kappa} \bigl( v(\sigma) \bigr) \leq n \, \max\{ \S_{\kappa} (\mu_0) , \S_{\kappa} (\mu_1)\} + A ,
\]
where the number
\[
A:= - \min\{ \S_{\kappa} (\mu_0),  \S_{\kappa} (\mu_1)\} + \max_{s\in [0,1]} \S_{\kappa} (\tau|_{[0,s]}) + \S_{\kappa} (\hat\tau) + \max_{s\in [0,1]} \S_{\kappa} (\gamma_s) 
\]
does not depend on $n$. The thesis follows.
\end{proof}

Let us consider again the case $\kappa_*\in P$. Fix a number $\kappa^*\in I$ such that $\kappa^*>\kappa_*$. If we apply the above lemma to $\mu_0 = \alpha_{\kappa_*}$, $\mu_1=\mu$ and $\kappa = \kappa^*$, by (\ref{s2}) we can find an integer $n_0$ such that
\[
p_n(\kappa^*) < 0 \qquad \forall n\geq n_0.
\]
By the monotonicity of $p_n$, we deduce that, up to the replacement of $I$ by the interval $I\cap (0,\kappa^*)$, we may assume that
\[
p_n(\kappa) < 0 \qquad \forall \kappa\in I, \forall n\geq n_0.
\]
Since $\alpha_{\kappa_*}^{n_0}$ is a strict local minimizer of $\S_{\kappa_*}$, by Lemma \ref{strict} we can find a bounded neighborhood $\mathcal{U}\subset \mathcal{M}$ of $\T\cdot \alpha_{\kappa_*}^{n_0}$ which does not contain $\mu^{n_0}$ and such that
\[
\inf_{\partial \mathcal{U}} \S_{\kappa_*} > \S_{\kappa_*} (\alpha_{\kappa_*}^{n_0}).
\]
Since $\S_{\kappa} \rightarrow \S_{\kappa_*}$ uniformly on bounded sets for $\kappa \rightarrow \kappa_*$, up to reducing the interval $I$ even more we may assume that 
\[
\inf_{\partial \mathcal{U}} \S_{\kappa} > \S_{\kappa} (\alpha_{\kappa_*}^{n_0}), \qquad \forall \kappa\in I.
\]
Since any path belonging to $\mathcal{P}_{n_0}$ must cross $\partial \mathcal{U}$, the above inequality implies that
\begin{equation}
\label{prima}
p_{n_0} (\kappa) > \S_{\kappa} (\alpha_{\kappa_*}^{n_0}), \qquad \forall \kappa\in I.
\end{equation}
Recalling also (\ref{s2}), we have proved the following:

\begin{Lemma}
\label{crita}
For any $\kappa_* \in P$ there is a natural number $n_0$ and an open interval $I\subset (0,c_u)$ containing $\kappa_*$ such that
\[
\S_{\kappa}(\mu^{n_0}) < \S_{\kappa} (\alpha_{\kappa_*}^{n_0}) < p_{n_0} (\kappa) < 0
\]
for every $\kappa$ in $I$.
\end{Lemma}

Now let $\kappa_*\in Q$. Applying Lemma \ref{bangert} to $\kappa = \kappa^*$, for some $\kappa^*\in I$, $\kappa^*>\kappa_*$, $\mu_0 = \tilde{\alpha}_{\kappa^*}$ and $\mu_1=\mu$, by (\ref{sotto}) we have
\[
\lim_{n\rightarrow \infty} q_n(\kappa^*) = -\infty.
\]
Since $q_n$ is increasing, up to the replacement of $I$ with the interval $I \cap (0,\kappa^*)$ we may assume that
\begin{equation}
\label{menoi}
\lim_{n\rightarrow \infty} q_n(\kappa) = -\infty \qquad \mbox{uniformly in } \kappa\in I,
\end{equation}
and in particular there exists an integer $n_0\geq 1$ such that
\[
q_n(\kappa) < 0 \qquad \forall \kappa\in I, \; \forall n\geq n_0.
\]
Let $\kappa\in I$. Since $\tilde{\alpha}_{\kappa}$ is a strict local minimizer, Lemma \ref{persistence} implies that $\tilde{\alpha}_{\kappa}^n$ is a strict local minimizer for every $n\in \N$. Therefore
\[
q_n(\kappa) > \S_{\kappa} (\tilde{\alpha}_{\kappa}^n) \qquad \qquad \forall \kappa\in I, \; \forall n\in \N.
\]
Notice that here, unlike in (\ref{prima}), the strict inequality between the minimax value and the value of the functional on one end-point of the paths in the mountain pass class holds for every $n$ and every $\kappa\in I$. In (\ref{prima}) instead, enlarging the set of $n$ for which the inequality is strict may force to reduce the interval $I$. Recalling also (\ref{sotto}), we have proved the following:

\begin{Lemma}
\label{critb}
Let $\kappa_*\in Q$. Then there is a natural number $n_0$ and an open interval $I\subset (0,c_u)$ containing $\kappa_*$ such that
\[
\S_{\kappa}(\mu^n) < \S_{\kappa} (\tilde{\alpha}_{\kappa}^n) < q_n (\kappa) < 0
\]
for every $\kappa$ in $I$ and every $n\geq n_0$.
\end{Lemma}

\section{The monotonicity argument} 

The following lemma uses Struwe's monotonicity argument \cite{str90} (see also \cite[Proposition 7.1]{con06} for an application of this argument to the free period action functional) and replaces the classical deformation lemma.

\begin{Lemma}
\label{mono}
Let $\kappa_*\in P$ (resp.\ $\kappa_* \in Q$). Let $n_0$ and $I$ be as in Lemma \ref{crita} (resp.\ \ref{critb}). Let $\bar\kappa\in I$ be a point where the function
\[
c(\kappa) := p_{n_0}(\kappa) \qquad \mbox{(resp. } c(\kappa) := q_n(\kappa) \mbox{ for some } n\geq n_0 \mbox{)}
\]
has a linear modulus of right-continuity, that is there exist $\delta>0$ and $M>0$ such that
\begin{equation}
\label{linmod}
c(\kappa) - c(\bar{\kappa}) \leq M (\kappa - \bar{\kappa} ), \quad \forall \kappa \in [\bar{\kappa},\bar{\kappa}+\delta)\subset I.
\end{equation}
Then for every open neihborhood $\mathcal{U}$ of the set
\[
\mathrm{crit}\, \S_{\bar\kappa} \cap \{\S_{\bar\kappa} = c(\bar\kappa)\} 
\]
there exists an element $v$ of $\mathcal{P}_{n_0}$ (resp.\ of $\mathcal{Q}_n(\bar\kappa)$) such that
\[
v([0,1]) \subset \{\S_{\bar\kappa} < c(\bar\kappa)\} \cup \mathcal{U}.
\]
\end{Lemma}

\begin{proof} 
Let $(\kappa_h)\subset I$ be a strictly decreasing sequence which converges to $\bar\kappa$, and set $\epsilon_h:=\kappa_h-\bar \kappa\downarrow0$. We pick $u_h$ in $\mathcal{P}_{n_0}$ (resp.\ in $\mathcal{Q}_n (\kappa_h)$) such that
\[
\max_{\sigma\in [0,1]}\S_{\kappa_h}\bigl( u_h(\sigma) \bigr) \leq c(\kappa_h)+\epsilon_h.
\]
Let $\gamma=(x,T)\in u_h([0,1])$ be such that $\S_{\bar \kappa}(\gamma)>c(\bar \kappa)-\epsilon_h$. Since $\bar\kappa$ satisfies (\ref{linmod}), we have
\[
T = \frac{\S_{\kappa_h}(\gamma)-\S_{\bar \kappa}(\gamma)}{\kappa_h-\bar\kappa} \leq \frac{c(\kappa_h)+\epsilon_h-c(\bar\kappa)+\epsilon_h}{\epsilon_h}\leq M+2.
\]
Moreover,
\[
\S_{\bar \kappa}(\gamma) \leq \S_{\kappa_h}(\gamma) \leq c(\kappa_h)+\epsilon_h \leq c(\bar \kappa)+(M+1)\epsilon_h.
\]
By the above considerations,
\begin{equation}
\label{dove1}
u_h\bigl( [0,1]\bigr) \subset \mathcal{A}_h \cup\big\{\S_{\bar \kappa}\leq c(\bar \kappa)-\epsilon_h\big\},
\end{equation}
where
\[
\mathcal{A}_h := \big\{(x,T) \in \mathcal{M} \,\big|\, T\leq M+2\mbox{ and }\S_{\bar \kappa} (x,T) \leq c(\bar \kappa)+(M+1)\epsilon_h\big\}.
\]
The estimate
\[
\S_{\bar\kappa}(x,T)  = \frac{1}{2T} \int_{\T} |\dot{x}(s)|^2\, ds + \int_{\T} x^*(\theta) + \bar\kappa T \geq \frac{1}{2T} \|\dot{x}\|_{L^2}^2 - \|\theta\|_{\infty} \|\dot{x}\|_{L^2}
\]
implies that the set $\mathcal{A}_h$ is bounded in $\mathcal{M}$, uniformly in $h$. Set
\[
\begin{split}
A_0 &:= \S_{\bar\kappa}(\alpha_{\kappa_*}^{n_0}) \quad \mbox{in the case }\kappa_*\in P, \\
A_0&:= \S_{\bar\kappa}(\tilde\alpha_{\bar\kappa}^{n}) \quad \mbox{in the case }\kappa_*\in Q.
\end{split}
\]
By  Lemmata \ref{crita} and \ref{critb} we can find numbers $A_1,A_2,A_3,A_4$ such that
\[
A_0 < A_1< A_2< c(\bar\kappa) < A_3 < A_4 <0. 
\]
In the case $\kappa_*\in P$ we have  $u_h(0)=\alpha_{\kappa_*}^{n_0}$ and $u_h(1)=\mu^{n_0}$, so Lemma \ref{crita} implies that
\begin{equation}
\label{estr1}
\S_{\bar{\kappa}}(u_h(0)) = A_0 < A_1, \quad \S_{\bar{\kappa}}(u_h(1)) < A_0 < A_1.
\end{equation}
In the case $\kappa_*\in Q$ the left-hand point of $u_h$ is instead $\tilde{\alpha}_{\kappa_h}^n$, but the continuity of $\kappa \rightarrow \tilde{\alpha}_{\kappa}^n$ guarantees the existence of $\eta>0$ such that
\begin{equation}
\label{dopo}
\S_{\bar\kappa}(\tilde{\alpha}_{\kappa}) < A_1 \qquad \forall \kappa\in (\bar\kappa - \eta,\bar\kappa+\eta),
\end{equation}
so also in this case
\begin{equation}
\label{estr2}
\S_{\bar{\kappa}}(u_h(0)) < A_1, \quad \S_{\bar{\kappa}}(u_h(1)) < A_0 < A_1,
\end{equation}
for $h$ large enough.

Let $\phi$ be the flow of the vector field obtained by multiplying $-\nabla\S_{\bar \kappa}$ by a suitable non-negative function, whose role is to make the vector field bounded on $\mathcal{M}$, vanishing on 
\[
\{ \S_{\bar\kappa} \leq A_1\} \cup \{ \S_{\bar\kappa} \geq A_4\},
\]
while keeping the uniform decrease condition
\begin{equation}
\label{ud}
\frac{d}{dr} \S_{\bar\kappa} \bigl(\phi_r(z)\bigr) \leq - \min \bigl\{
\|d\S_{\bar\kappa} (\phi_{r}(z)) \|^2,1 \bigr\}, \quad \mbox{if } A_2 \leq  \S_{\bar\kappa} (\phi_{r}(z))  \leq A_3.
\end{equation}
Then $\phi$ is well-defined on $[0,+\infty[\times \mathcal{M}$: the only source of non-completeness is $T$ going to 0, which by Lemma \ref{gon1} happens only for negative-gradient flow lines for which the action tends to 0; but we have made the vector field vanish near level 0. Since $\S_{\bar\kappa}$ decreases along the flow lines and $\phi$ maps bounded sets into bounded sets, we have from (\ref{dove1})
\begin{equation}
\label{qua}
\phi([0,1]\times u_h([0,1])) \subset \mathcal{B}_h \cup \bigl\{\S_{\bar \kappa}\leq c(\bar \kappa)-\epsilon_h \bigr\},
\end{equation}
for some uniformly bounded set
\begin{equation}
\label{bdd}
\mathcal{B}_h\subset \bigl\{\S_{\bar \kappa}\leq c(\bar \kappa)+(M+1)\epsilon_h \bigr\}.
\end{equation}
Let $\mathcal{B}\subset \mathcal{M}$ be a bounded closed set which contains $\mathcal{B}_h$ for every $h\in \N$. Since the Palais-Smale condition holds on bounded sets (Lemma \ref{gon2}), the set 
\[
K:= \mathcal{B}\cap \mathrm{crit}\, \S_{\bar\kappa} \cap\{ \S_{\bar\kappa} = c(\bar\kappa)\}
\]
is compact. The open set $\mathcal{U}$ is in particular a neighborhood of $K$.
Since $K$ consists of fixed points of the flow $\phi$, we can find an open neighoborhood $\mathcal{V}\subset \mathcal{U}$ of $K$ such that 
\begin{equation}
\label{disc}
\phi([0,1]\times \mathcal{V}) \subset \mathcal{U}.
\end{equation}
Since $\S_{\bar\kappa}$ satisfies the Palais-Smale condition on $\mathcal{B}$, we can find $\epsilon>0$ and $0<\delta\leq 1$ such that
\begin{equation}
\label{ps}
\|d\S_{\bar\kappa}(\gamma)\| \geq \delta \qquad \forall \gamma \in (\mathcal{B} \setminus \mathcal{V}) \cap \{c(\bar\kappa) - \epsilon \leq \S_{\bar\kappa} \leq c(\bar\kappa) + \epsilon\}.
\end{equation}
By (\ref{qua}) and (\ref{bdd}), for every $(r,\sigma)\in [0,1]\times [0,1]$ and every $h\in \N$ there holds 
\begin{equation}
\label{dove}
\phi_r(u_h(\sigma)) \subset  \bigl( \mathcal{B} \cap \{ \S_{\bar\kappa} \leq c(\bar\kappa) + (M+1)\epsilon_h \} \bigr) \cup \{\S_{\bar\kappa}<c(\bar\kappa)\}.
\end{equation}
Let $\sigma\in [0,1]$ be such that 
\begin{equation}
\label{viola}
\S_{\bar\kappa}(\phi_1(u_h(\sigma)))\geq c(\bar\kappa) \quad \mbox{and} \quad \phi_1(u_h(\sigma)) \notin \mathcal{U}.
\end{equation}
By (\ref{disc}), $\phi_r(u_h(\sigma))$ cannot belong to $\mathcal{V}$ for any $r\in [0,1]$, and together with (\ref{dove}) and the fact that $\S_{\bar\kappa}$ decreases along the orbits of $\phi$, we obtain
\[
\phi([0,1]\times \{u_h(\sigma)\}) \subset (\mathcal{B}\setminus \mathcal{V}) \cap \{ c(\bar\kappa) \leq \S_{\bar\kappa} \leq  c(\bar\kappa) + (M+1)\epsilon_h\}.
\]
If $h$ is so large that $(M+1)\epsilon_h \leq \epsilon$, (\ref{ps}) implies that 
\[
d\S_{\bar\kappa} (\phi_r(u_h(\sigma))) \geq \delta \qquad \forall r\in [0,1],
\]
so by (\ref{ud})
\[
\begin{split}
c(\bar\kappa) \leq \S_{\bar\kappa}(\phi_1(u_h(\sigma))) &= \S_{\bar\kappa}(\phi_0(u_h(\sigma))) + \int_0^1 \frac{d}{dr} \S_{\bar\kappa}(\phi_r(u_h(\sigma))) \, dr \\ &\leq c(\bar\kappa) + (M+1)\epsilon_h - \delta^2.
\end{split}
\]
The above inequality implies that $\epsilon_h \geq \delta^2/(M+1)$. Since $(\epsilon_h)$ is infinitesimal, this gives us an upper bound on $h$: when $h$ is larger than this upper bound, then there cannot be any $\sigma\in [0,1]$ such that (\ref{viola}) holds. For such a large $h$ we must hence have
\begin{equation}
\label{scopo}
\phi_1(u_h (\sigma)) \in \{\S_{\bar\kappa} < c(\bar\kappa)\} \cup \mathcal{U} \qquad \forall \sigma \in [0,1].
\end{equation}

Let $h$ be a large natural number, so that (\ref{scopo}) holds. In the case $\kappa_*\in Q$ we also require $h$ to be so large that $\kappa_h < \bar\kappa + \eta$, where $\eta$ is the positive constant appearing in (\ref{dopo}). In the case $\kappa_*\in P$ we define
\[
v := \phi_1 \circ u_h,
\]
and we observe that, since the flow $\phi$ fixes the points in $\{\S_{\bar\kappa}\leq A_1\}$, (\ref{estr1}) implies that $v$ belongs to $\mathcal{P}_n$.

In the case $\kappa_*\in Q$ the end points of the path $u_h\in \mathcal{Q}_n(\kappa_h)$ are also fixed by the flow $\phi$ because of (\ref{estr2}), and we obtain an element of $\mathcal{Q}_n(\bar\kappa)$ by setting
\[
v(\sigma) := \left\{ \begin{array}{ll} \tilde{\alpha}_{\bar\kappa + 2 \sigma(\kappa_h - \bar\kappa)} & \mbox{for } 0\leq \sigma \leq 1/2, \\ \phi_1(u_h(2\sigma-1)) & \mbox{for } 1/2< \sigma \leq 1. \end{array} \right.
\]
Since for $\kappa\in [\bar\kappa, \kappa_h]$ 
\[
\S_{\bar{\kappa}} (\tilde{\alpha}_{\kappa}) < A_1 < c(\bar\kappa)
\]
by (\ref{dopo}), we conclude that in both cases
\[
v([0,1]) \subset \{\S_{\bar\kappa} < c(\bar\kappa)\} \cup \mathcal{U}.
\]
\end{proof} 

\section{The proof of the Theorem}

Lemma \ref{mono} has the following consequence, where we incorporate a result about the nature of the set of critical points found by a mountain pass minimax which, under the assumption that the Palais-Smale condition holds, is due to H.~Hofer (see \cite{hof85} and \cite{hof86}). 

\begin{Lemma} 
\label{mpass}
Let $\kappa_* \in P$ (resp.\ $\kappa_*\in Q$). Let $n_0$ and $I$ be as in Lemma \ref{crita} (resp. \ref{critb}). Then for almost every $\bar{\kappa}$ in $I$ the number $c := p_{n_0} (\bar\kappa)$ (resp.\ $c:=q_n(\bar\kappa)$ for any $n \geq n_0$) is a critical value of $\S_{\bar\kappa}$. Furthermore, every open neighborhood $\mathcal{U}$ of 
\[
\mathrm{crit}\, \S_{\bar\kappa} \cap \{\S_{\bar\kappa} = c\} 
\]
has a connected component $\mathcal{U}_0$ such that the set
\[
\mathcal{U}_0 \cap \{\S_{\bar\kappa} < c\}
\]
is non-empty and not connected. In particular, there is a critical point $\beta$ of $\S_{\bar\kappa}$ of action $\S_{\bar\kappa}(\beta)=c$ which is not a strict local minimizer. If such a $\beta$ is transversally non-degenerate then it has Morse index 1.
\end{Lemma}

\begin{proof}
Let $J\subset I$ be the set of points at which $p_{n_0}$ is differentiable (resp. $q_n$ is differentiable for every $n\geq n_0$). Since $p_{n_0}$ (resp.\ $q_n$) is a monotone function, $J$ has full measure in $I$ (in the case $\kappa_* \in Q$ we also use the fact that a countable intersection of sets of full measure has full measure).

Let $\bar\kappa\in J$. Then $c = p_{n_0} (\bar\kappa)$ (resp.\ $c=q_n(\bar\kappa)$ for some $n \geq n_0$) is a critical value of $\S_{\bar\kappa}$: if not we can take $\mathcal{U}=\emptyset$ in Lemma \ref{mono} and we find $v$ in $\mathcal{P}_{n_0}$ (resp.\ $\mathcal{Q}_n(\bar\kappa)$) such that
\[
v([0,1]) \subset \{\S_{\bar\kappa} < c\},
\]
thus contradicting the definition of $c=p_n(\bar\kappa)$ (resp.\ $c=q_n(\bar\kappa)$).

Now let $\mathcal{U}$ be an open neighborhood of 
\[
\mathrm{crit}\, \S_{\bar\kappa} \cap \{\S_{\bar\kappa} = c\} 
\]
and assume by contradiction that for each connected component $\mathcal{U}_0$ of $\mathcal{U}$ the open set 
\[
\mathcal{U}_0 \cap \{\S_{\bar\kappa} < c\}
\]
is either empty or connected. By Lemma \ref{mono} we can find $v$ in $\mathcal{P}_{n_0}$ (resp.\ $\mathcal{Q}_n(\bar\kappa)$) such that
\[
v([0,1]) \subset \{\S_{\bar\kappa} < c\} \cup \mathcal{U}.
\]
Consider the open set 
\[
U:= v^{-1}(\mathcal{U}) \subset [0,1].
\]
Since $\S_{\bar\kappa} \circ v<c$ on the compact set $I\setminus U$, we 
have 
\[
\max_{\sigma\in I\setminus U}   \S_{\bar\kappa} \circ v < c.
\]
Since the function $\S_{\bar\kappa} \circ v$ is uniformly continuous on $[0,1]$, we can find a positive number $\delta$ such that if $U_0$ is a connected component of $U$ with length less than $\delta$ then $\S_{\bar\kappa} \circ v < c$ on $U_0$. Therefore, there are at most finitely many connected components $U_1,U_2,\dots,U_k$ of $U$ where the supremum of $\S_{\bar\kappa} \circ v$ is at least $c$. Then we can find numbers
\[
0 < a_1<b_1 < a_2<b_2 < \dots < a_k<b_k <1
\]
such that
\[
[a_j,b_j] \subset U_j \quad \forall j=1,2,\dots,k
\]
and
\[
\max_{[0,1] \setminus \bigcup_{j=1}^k (a_j,b_j)} \S_{\bar\kappa} \circ v < c.
\]
Since $v(a_j)$ and $v(b_j)$ belong to the same connected component of $\mathcal{U}$, our assumption implies that there exists a continuous path $v_j: [a_j,b_j] \rightarrow \mathcal{M}$ such that $v_j(a_j)=v(a_j)$, $v_j(b_j)=v(b_j)$, and 
\[
v_j([a_j,b_j]) \subset \mathcal{U} \cap \{\S_{\bar\kappa} < c\}.
\]
Therefore the path
\[
w(\sigma) := \left\{ \begin{array}{ll} v(\sigma) & \mbox{for } \sigma\in [0,1] \setminus  \bigcup_{j=1}^k (a_j,b_j) \\ v_j(\sigma) & \mbox{for } \sigma\in [a_j,b_j], \; j=1,2,\dots,k, \end{array} \right.
\]
is in $\mathcal{P}_{n_0}$ (resp.\ $\mathcal{Q}_n(\bar\kappa)$) and satifies
\[
\max_{[0,1]} \S_{\bar\kappa} \circ w < c.
\]
This contradicts the definition of $c=p_n(\bar\kappa)$ (resp.\ $c=q_n(\bar\kappa)$) and implies that there is at least one component of $\mathcal{U}$ whose intersection with $\{\S_{\bar\kappa}<c\}$ is non-empty and not connected.

In particular, $\mathrm{crit} \, \S_{\bar\kappa} \cap \{\S_{\bar\kappa}=c\}$ contains a point $\beta$ which not a strict local minimizer: if this set consists of strict local minimizers, then it has an open neighborhood $\mathcal{U}$ such that $\mathcal{U} \cap \{\S_{\bar\kappa}<c\}$ is empty.
 
When $\beta$ is transversally non-degenerate, the connected component of $\mathrm{crit}\, \S_{\bar\kappa}\cap \{\S_{\bar\kappa} = c\}$ containing $\beta$ is of the form $\T\cdot \beta$, and the Morse-Bott Lemma implies that this component consists of critical points of Morse index 1: Morse-Bott components of index zero have neighborhoods $\mathcal{U}_0$ such that $\mathcal{U}_0 \cap \{\S_{\bar\kappa} < c\}$ is empty, while Morse-Bott components of index larger than one have neighborhoods $\mathcal{U}_0$ such that $\mathcal{U}_0 \cap \{\S_{\bar\kappa} < c\}$ is connected.
\end{proof}

\begin{Remark}
\label{stability}
Since the Morse index of $\beta$ is one if it is transversally non-degenerate,  Proposition \ref{elliptic-odd_hyp} implies that the eigenvalues $\lambda,1/\lambda$ of the linearized Poincar\'e map $P$ of  the closed magnetic geodesic $\beta$ are either on the unit circle (that is, $\beta$ is non-hyperbolic), or in $(-\infty,0)\setminus \{-1\}$ ($\beta$ is odd hyperbolic).
\end{Remark} 

Let $R\subset (0,c_u)$ be the set of values $\kappa$ for which the energy level $E^{-1}(\kappa)$ is non-degenerate. Notice that
\[
R \subset Q \subset P \subset (0,c_u).
\]
The above lemma, together with known results, has the following consequence: 

\begin{Lemma}
\label{main}
Let $\kappa$ be an energy level in $(0,c_u)$. Then:
\begin{enumerate}
\item if $\kappa\in P^c=(0,c_u) \setminus P$ then there are infinitely many closed magnetic geodesics of energy $\kappa$;
\item every $\kappa\in P$ is contained in an open interval $I(\kappa)\subset (0,c_u)$ in which there is a set $I_0(\kappa)$ of full measure such that for every $\kappa_0\in I_0(\kappa)$ there are at least three closed magnetic geodesics of energy $\kappa_0$;
\item every $\kappa\in Q$ is contained in an open interval $J(\kappa)\subset (0,c_u)$ in which there is a set $J_0(\kappa)$ of full measure such that for every $\kappa_0\in J_0(\kappa)\cap R$ there are infinitely many closed magnetic geodesics of energy $\kappa_0$.
\end{enumerate}
\end{Lemma}

\begin{proof}
(i) If $\kappa\in P^c$, then the local minimizer $\alpha_{\kappa}$ is not strict, hence $\S_{\kappa}$ has a sequence of local minimizers in $\mathcal{M} \setminus \T \cdot \alpha_{\kappa}$
which converges to $\alpha_{\kappa}$. In particular, there are infinitely many closed magnetic geodesics of energy $\kappa$, proving (i).

\medskip

(ii)  G.\ Contreras has proved in \cite{con06} that for almost every $\kappa$ in $(0,c_u)$ there exists a closed contractible orbit $\gamma_{\kappa}$ with energy $\kappa$ and positive $\S_{\kappa}$-action (for more general systems on arbitrary compact configuration spaces, of any dimension). In particular, $\gamma_{\kappa}$ is geometrically distinct from $\alpha_{\kappa}$, all of whose iterates have negative $\S_{\kappa}$-action.

By Lemma \ref{mpass}, $\kappa\in P$ is contained in an open interval $I(\kappa)\subset (0,c_u)$ such that for almost every $\kappa_0\in I(\kappa)$ the functional $\S_{\kappa_0}$ has a critical point $\beta$ with action $\S_{\kappa_0}(\beta) = p_{n_0}(\kappa_0)<0$ which is not a strict local minimizer. In the case $\kappa_0\in P^c$, there are infinitely many closed magnetic geodesics of energy $\kappa_0$ by (i). In the case $\kappa_0\notin P$, the closed magnetic geodesic $\beta$ cannot coincide with an iterate of $\alpha_{\kappa_0}$, because all such iterates are strict local minimizers. Moreover, it cannot coincide with an iterate of $\gamma_{\kappa_0}$ because these iterates have positive $\S_{\kappa_0}$-action. 
Therefore for almost every $\kappa_0\in I(\kappa)$ there are at least three geometrically distinct closed magnetic geodesics of energy $\kappa_0$.

\medskip

(iii) By Lemma \ref{mpass}, $\kappa\in Q$ is contained in an open interval $J(\kappa)\subset (0,c_u)$ which has a subset $J_0(\kappa)$ of full measure with the following property: for every $\kappa_0\in J_0(\kappa)\cap R$ and every $n\geq n_0$ the functional $\S_{\kappa_0}$ has a critical point  $\beta_{\kappa_0,n}$ of action $\S_{\kappa_0}(\beta_{\kappa_0,n}) = q_n(\kappa_0)$ and Morse index $i(\beta_{\kappa_0,n})=1$. We claim that the closed magnetic geodesics $\beta_{\kappa_0,n}$, $n\geq n_0$, cannot be the iterates of only finitely many closed magnetic geodesics. Indeed, if by contradiction this is the case, we can find a closed magnetic geodesic $\beta$ and a sequence of integers $(m_h)$, $m_h \geq n_0$ and $m_h \rightarrow +\infty$, such that $\beta_{\kappa_0,m_h}$ is the iterate $\beta^{k_h}$, for some sequence of integers $k_h\geq 1$. Since the sequence
\[
q_{m_h}(\kappa_0) = \S_{\kappa_0} ( \beta_{\kappa_0,m_h}) = k_h \,\S_{\kappa_0} (\beta)
\]
tends to $-\infty$ (see (\ref{menoi})), the sequence $(k_h)$ must diverge to $+\infty$. 
Therefore, $\beta$ has mean index zero:
\[
\widehat{\imath}(\beta) = \lim_{h\rightarrow \infty} \frac{i(\beta^{k_h})}{k_h} = \lim_{h\rightarrow \infty} \frac{i(\beta_{\kappa_0,m_h})}{k_h} = \lim_{h\rightarrow \infty} \frac{1}{k_h} = 0.
\]
On the other hand, Theorem \ref{index} implies that $\beta_{\kappa_0,h}$ has positive mean index, and hence
\[
\widehat{\imath}(\beta) = \frac{1}{k_h} \widehat{\imath} (\beta_{\kappa_0,m_h}) > 0.
\]
This contradiction proves that the set $\set{\beta_{\kappa_0,n}}{n\geq n_0}$ consists of infinitely many geometrically distinct closed magnetic geodesics of energy $\kappa_0$. 
\end{proof}

\begin{proof}[Proof of the Theorem.] 
The Theorem stated in the Introduction follows from Lemma \ref{main} purely by set- and measure-theoretic arguments. Indeed, for every $\kappa\in P$ let $I(\kappa)$ and $I_0(\kappa)$ be as in Lemma \ref{main} (ii). Since the topology of $\R$ admits a countable basis, there exists an at most countable subset $P_0$ of $P$ such that
\begin{equation}
\label{incl}
P \subset \bigcup_{\kappa\in P_0} I(\kappa).
\end{equation}
By Lemma \ref{main} (i) and (ii), the energy level $E^{-1}(\kappa)$ admits at least three magnetic geodesics whenever $\kappa$ belongs to the set
\[
K_3:= P^c \cup \bigcup_{k\in P} I_0(\kappa),
\]
where $P^c$ denotes the complement of $P$ in $(0,c_u)$. We must show that the above set has full measure in $(0,c_u)$, that is that its complement
\[
P \cap \Bigl( \bigcup_{k\in P} I_0(\kappa) \Bigr)^c
\]
has measure zero. By (\ref{incl}), the above set is contained in
\[
\Bigl(\bigcup_{\kappa\in P_0} I(\kappa)\Bigr) \cap  \Bigl( \bigcup_{k\in P} I_0(\kappa) \Big)^c,
\]
which is clearly contained in
\[
\bigcup_{\kappa\in P_0} \bigl( I(\kappa) \setminus I_0(\kappa) \bigr).
\]
The above set has measure zero, being an at most countable union of sets with measure zero. This proves that $K_3$ has full measure in $(0,c_u)$.

Now let $J(\kappa)$ and $J_0(\kappa)$ be the sets given by Lemma \ref{main} (iii), and let $Q_0$ be an at most countable subset of $Q$ such that
\begin{equation}
\label{incl2}
R\subset  Q \subset \bigcup_{\kappa\in Q_0} J(\kappa).
\end{equation}
By Lemma \ref{main} (iii)  the energy level $E^{-1}(\kappa)$ admits infinitely many magnetic geodesics whenever $\kappa$ belongs to the set
\[
K_{\infty}:=  \bigcup_{k\in Q} \bigl( J_0(\kappa)\cap R\bigr) \subset R.
\]
We must prove that $K_{\infty}$ has full measure in $R$, that is that $R\setminus K_{\infty}$ has measure zero. We have the inclusion
\[
R \setminus K_{\infty} =  R \setminus \bigcup_{k\in Q} \bigl( J_0(\kappa)\cap R\bigr)  \subset  R \setminus \bigcup_{k\in Q_0}  J_0(\kappa),
\]
from which, together with (\ref{incl2}), we obtain
\[
R \setminus K_{\infty} \subset \Bigl( \bigcup_{\kappa\in Q_0} J(\kappa) \Bigr) \setminus \Bigl( \bigcup_{k\in Q_0}  J_0(\kappa) \Bigr) \subset \bigcup_{\kappa\in Q_0} \bigl( J(\kappa) \setminus J_0(\kappa) \bigr).
\]
Therefore, $R\setminus K_{\infty}$ has measure zero. 

We conclude that the set
\[
K:= K_3 \cap (R^c \cup K_{\infty}) \subset (0,c_u)
\]
satisfies the requirements of the Theorem. Indeed, $K$ has full measure in $(0,c_u)$ because $K_3^c$ and 
\[
(R^c \cup K_{\infty})^c = R \cap K_{\infty}^c = R \setminus K_{\infty}
\]
have measure zero. Being a subset of $K_3$, $K$ consists of energy levels for which there at least three closed magnetic geodesics. From the inclusion
\[
R \cap K \subset R \cap (R^c \cup K_{\infty}) = R \cap K_{\infty} = K_{\infty} 
\]
it follows that for every energy level in $R\cap K$ there are infinitely many closed magnetic geodesics.
\end{proof}


\begin{thebibliography}{CIPP00}

\bibitem[Abb01]{abb01}
A.~Abbondandolo.
\newblock {\em Morse theory for {H}amiltonian systems}, volume 425 of {\em
  Pitman Research Notes in Mathematics}.
\newblock Chapman \& Hall, London, 2001.

\bibitem[AM78]{am78}
R.~Abraham and J.~E. Marsden.
\newblock {\em Foundations of mechanics}.
\newblock Advanced Book Program, Reading, Mass. Benjamin/Cummings Publishing
  Co., Inc., 1978.
  
\bibitem[Arn61]{arn61b}
V.~I. Arnold.
\newblock Some remarks on flows of line elements and frames.
\newblock {\em Dokl. Akad. Nauk. SSSR}, 138:255--257, 1961.

\bibitem[Ban80]{ban80}
V.~Bangert.
\newblock Closed geodesics on complete surfaces.
\newblock {\em Math. Ann.}, 251:83--96, 1980.

\bibitem[BK83]{bk83}
V.~Bangert and W.~Klingenberg.
\newblock Homology generated by iterated closed geodesics.
\newblock {\em Topology}, 22:379--388, 1983.

\bibitem[BL10]{bl10}
V.~Bangert and Y.~Long.
\newblock The existence of two closed geodesics on every {F}insler 2-sphere.
\newblock {\em Math. Ann.}, 346:335--366, 2010.

\bibitem[BN93]{bn93}
H.~Brezis and L.~Nirenberg.
\newblock {$H^1$ versus $C^1$ local minimizers}.
\newblock {\em C. R. Acad. Sci. Paris S\'er. I Math.}, 317:465--472, 1993.

\bibitem[Bot56]{bot56}
R.~Bott.
\newblock On the iteration of closed geodesics and the {S}turm intersection
  theory.
\newblock {\em Comm. Pure Appl. Math.}, 9:171--206, 1956.

\bibitem[Cha94]{cha94}
K.~C. Chang.
\newblock {$H^1$ versus $C^1$ isolated critical points}.
\newblock {\em C. R. Acad. Sci. Paris S\'er. I Math.}, 319:441--446, 1994.

\bibitem[CIPP00]{cipp00}
G.~Contreras, R.~Iturriaga, G.~P. Paternain, and M.~Paternain.
\newblock The {P}alais-{S}male condition and {M}a\~n\'e's critical values.
\newblock {\em Ann. Henri Poincar\'e}, 1:655--684, 2000.

\bibitem[CMP04]{cmp04}
G.~Contreras, L.~Macarini, and G.~P. Paternain.
\newblock Periodic orbits for exact magnetic flows on surfaces.
\newblock {\em Internat. Math. Res. Notices}, 8:361--387, 2004.

\bibitem[Con06]{con06}
G.~Contreras.
\newblock The {P}alais-{S}male condition on contact type energy levels for
  convex {L}agrangian systems.
\newblock {\em Calc. Var. Partial Differential Equations}, 27:321--395, 2006.

\bibitem[CZ84]{cz84}
C.~Conley and E.~Zehnder.
\newblock {M}orse-type index theory for flows and periodic solutions for
  {H}amiltonian equations.
\newblock {\em Comm. Pure Appl. Math.}, 37:207--253, 1984.

\bibitem[Hed32]{H} G.A. Hedlund. \newblock Geodesics on a two-dimensional Riemannian manifold with periodic coefficients. 
\newblock {\em Ann. of Math. (2),} 33:719--739, 1932.

\bibitem[Hof85]{hof85}
H.~Hofer.
\newblock A geometric description of the neighborhood of a critical point given
  by the mountain-pass theorem.
\newblock {\em J. London Math. Soc. (2)}, 31:566--570, 1985.

\bibitem[Hof86]{hof86}
H.~Hofer.
\newblock The topological degree at a critical point of mountain-pass type.
\newblock In {\em Nonlinear functional analysis and its applications, Part 1
  (Berkeley, Calif., 1983)}, volume 45, Part 1 of {\em Proc. Sympos. Pure
  Math.}, pages 501--509, Providence, RI, 1986. Amer. Math. Soc.
  
\bibitem[HZ94]{hz94}
 H.~Hofer and E.~Zehnder.
\newblock Symplectic Invariants and Hamiltonian Dynamics.
\newblock Birkh\"auser Advanced Texts, Basel, 1994.

\bibitem[Kat73]{kat73}
A.~Katok.
\newblock Ergodic perturbations of degenerate integrable {H}amiltonian systems.
\newblock {\em Izv. Akad. Nauk SSSR Ser. Mat.}, 37:539--576, 1973.

\bibitem[KH95]{KH} A. Katok and B. Hasselblatt.
\newblock Introduction to the modern theory of dynamical systems. With a supplementary chapter by Katok and Leonardo Mendoza. 
\newblock Encyclopedia of Mathematics and its Applications, 54. Cambridge University Press, Cambridge, 1995.

\bibitem[Lon99]{lon99}
Y.~Long.
\newblock Bott formula of the {M}aslov-type index theory.
\newblock {\em Pacific J. Math.}, 187:113--149, 1999.

\bibitem[Lon02]{lon02}
Y.~Long.
\newblock {\em Index theory for symplectic paths with applications}.
\newblock Birkh\"auser, Basel, 2002.

\bibitem[Maz11]{maz11}
M.~Mazzucchelli.
\newblock {\em Critical point theory for {L}agrangian systems}, volume 293 of
  {\em Progress in Mathematics}.
\newblock Birkh\"auser, 2011.

\bibitem[Mir06]{mir06}
J.~A.~G. Miranda.
\newblock Generic properties for magnetic flows on surfaces.
\newblock {\em Nonlinearity}, 19:1849--1874, 2006.

\bibitem[Mir07]{mir07}
J.~A.~G. Miranda.
\newblock Positive topological entropy for magnetic flows on surfaces.
\newblock {\em Nonlinearity}, 20:2007--2031, 2007.

\bibitem[MP11]{mp11}
W.~J. Merry and G.~P. Paternain.
\newblock Index computations in {R}abinowitz {F}loer homology.
\newblock {\em J. fixed point theory appl.}, 10:87--111, 2011.

\bibitem[PP97]{pp97}
G.~P. Paternain and M.~Paternain.
\newblock Critical values of autonomous {L}agrangian systems.
\newblock {\em Comment. Math. Helv.}, 72(3):481--499, 1997.

\bibitem[Str90]{str90}
M.~Struwe.
\newblock Existence of periodic solutions of {H}amiltonian systems on almost
  every energy surface.
\newblock {\em Bol. Soc. Bras. Mat.}, 20:49--58, 1990.

\bibitem[Syc08]{syc08}
M.~A. Sychev.
\newblock Local minimizers of one-dimensional variational problems and obstacle
  problems.
\newblock {\em C. R. Acad. Sci. Paris S\'er. I Math.}, 346:1213--1218, 2008.

\bibitem[Tai92a]{tai92b}
I.~A. Taimanov.
\newblock Closed extremals on two-dimensional manifolds.
\newblock {\em Russian Math. Surveys}, 47:163--211, 1992.

\bibitem[Tai92b]{tai92c}
I.~A. Taimanov.
\newblock Closed non self-intersecting extremals of multivalued functionals.
\newblock {\em Siberian Math. J.}, 33:686--692, 1992.

\bibitem[Tai92c]{tai92}
I.~A. Taimanov.
\newblock Non self-intersecting closed extremals of multivalued or
  not-everywhere-positive functionals.
\newblock {\em Math. USSR-Izv.}, 38:359--374, 1992.

\bibitem[Ust99]{ust99}
I.~Ustilovsky.
\newblock Contact homology and contact structures on $S^{4m+1}$.
\newblock Ph.D. thesis, Stanford University, 1999.

\end{thebibliography}

\end{document}